\documentclass[11pt]{article}

\usepackage[latin1]{inputenc}
\usepackage[T1]{fontenc}
\usepackage{fancyhdr}
\usepackage{amsfonts}
\usepackage{amsmath}
\usepackage{amssymb}
\usepackage{amsthm}
\usepackage{cases}
\usepackage{mathrsfs}
\usepackage{amscd}
\usepackage{amstext}
\usepackage{tikz}
\usepackage{tikz-cd}

\usepackage{paralist}

\newlength{\fixboxwidth}
\newcommand{\re}{\mathbb{R}}\newcommand{\N}{\mathbb{N}}
\newcommand{\zz}{\mathbb{Z}}\newcommand{\C}{\mathbb{C}}

\newcommand{\Z}{{\zz}^d}

\newcommand{\R}{{\re}^d}

\newcommand{\ce}{{\mathcal E}}

\newcommand{\cf}{{\mathcal F}}
\newcommand{\cfi}{{\cf}^{-1}}

\newcommand{\supp}{{\rm supp \, }}

\newcommand{\bproof}{\begin{proof}}
\newcommand{\eproof}{\end{proof}}

\newcommand{\be}{\begin{equation}}
\newcommand{\ee}{\end{equation}}
\newcommand{\beq}{\begin{eqnarray}}
\newcommand{\beqq}{\begin{eqnarray*}}
\newcommand{\eeq}{\end{eqnarray}}
\newcommand{\eeqq}{\end{eqnarray*}}

\numberwithin{equation}{section}

\newtheorem{theorem}{Theorem}[section]
\newtheorem{definition}[theorem]{Definition}
\newtheorem{corollary}[theorem]{Corollary}
\newtheorem{lemma}[theorem]{Lemma}
\newtheorem{proposition}[theorem]{Proposition}
\newtheorem{remark}[theorem]{Remark}

\begin{document}

\nocite{*}

\title{Bernstein Numbers of Embeddings of Isotropic and Dominating Mixed Besov Spaces}

\author{Van Kien Nguyen\\
Friedrich-Schiller-University Jena, Ernst-Abbe-Platz 2, 07737 Jena, Germany\\
kien.nguyen@uni-jena.de
}


\date{\today}

\maketitle 

\begin{abstract}
The purpose of the present paper is to investigate the decay of  Bernstein numbers of the embedding from $B^t_{p_1,q}((0,1)^d)$ into the space $L_{p_2}((0,1)^d) $. The asymptotic behaviour of Bernstein numbers of the identity $id: S_{p_1,p_1}^tB((0,1)^d)\rightarrow L_{p_2}((0,1)^d)$ will be also considered. $ $ 

\end{abstract}
\section{Introduction} 
Bernstein widths were first introduced by Tikhomirov \cite{Ti}, see also \cite[Chapter 3, p. 187]{Ti2}. For a symmetric subset $K$ of a normed linear space $X$, the Bernstein $n$-width of the set $K$ is the quantity
$$b_n(K,X)=\sup_{L_{n+1}}\sup\{\lambda\geq 0: (\lambda U(X))\cap L_{n+1}\subset K\},$$
here $U(X)$ is the closed unit ball in $X$ and the outer supremum is taken over all subspaces $L_{n+1}$ of $X$ such that $\text{dim}\, L_{n+1} = n +1$. If $T$ is a continuous linear operator from Banach space $X$ to Banach space $Y$, that is $T\in \mathcal{L}(X,Y)$, then the $n$-th Bernstein number of $T$ is defined to be
$$ b_n(T)=b_n(T: X\to Y)=\sup_{L_n}\inf_{\substack{x\in L_n
\\ x\not =0}} \dfrac{\|Tx\|}{\| x\|} ,$$
where the supremum is taken over all subspaces $L_n$ of $X$ with dimension $n$. It is obvious that if $T$ is an injective mapping we have $b_{n+1}(T: X\to Y)=b_n(T(U(X)),Y)$. $ $

Bernstein numbers have some connections with $s$-numbers. According to Pietsch, see \cite[2.2.1]{Pi21}, an $s$-function is a map $s$ assigning to every operator $T\in \mathcal L(X,Y)$ a scalar sequence $(s_n(T))$ such that the following conditions are satisfied:
\begin{description}
\item[s1](monotonicity) $\|T\|=s_1(T)\geq s_2(T)\geq...\geq 0 $ for all $T\in \mathcal L(X,Y)$,
\item[s2] (additivity) $ s_{n+m-1}(S+T)\leq s_n(S)+ s_m(T) $ for $S,T\in \mathcal L(X,Y)$ and $m,n\in \mathbb{N}$,
\item[s3](ideal property) $s_n(BTA)\leq \|B\|\cdot s_n(T)\cdot \|A\|$ for $A\in \mathcal L(X_0,X)$, $T\in \mathcal L(X,Y)$, $B\in \mathcal L(Y,Y_0)$,
\item[s4](rank property) $s_n(T)=0$ if $\text{rank}(T)<n$,
\item[s5](norming property) $s_n(id: \ell_2^n\to \ell_2^n)=1$.
\end{description}
The scalar $s_n(T)$ is called the $n$th $s$-number of the operator $T$. Some examples of $s$-number are approximation, Kolmogorov, Gelfand and Weyl numbers, see \cite[Chapter 2]{Pi21}. An $s$-number is called multiplicative if it satisfies
\begin{description}
\item[s6] $s_{n+m-1}(ST)\leq s_n(S)\, s_m(T) $ for $T\in \mathcal L(X,Y)$, $S\in \mathcal L(Y,Z)$ and $m,n=1,2, \ldots \, $.
\end{description}
It is obvious that Bernstein numbers satisfy (s1), (s3), (s4) and (s5). However the question whether Bernstein numbers are $s$-numbers was open until 2008. The negative answer was given by Pietsch \cite{Pie2}. In this paper Pietsch gave an example proving that Bernstein numbers fail to be additive and multiplicative. In the same paper, he also showed that the classes 
$$\mathcal{L}_p^{bern}=\Big\{T: \sum_{n=1}^{\infty}b_n(T)^p<\infty\Big\}\ \ \text{ and }\ \ \mathcal{L}_{p,\infty}^{bern}=\Big\{T: \sup_{n\in \mathbb{N}}n^{\frac{1}{p}}b_n(T)<\infty\Big\},\ \ \ p\in (0,\infty),$$
fail to be operator ideals. There are several different versions of the notion $s$-number in the literature. In one of the earliest definition $(s_2)$ was replaced by
\begin{description}
\item[s2'.] $s_n(S+T)\leq s_n(T)+\|S\|$ for $S,T\in \mathcal L(X,Y)$ and $m,n\in \mathbb{N}$,
\end{description}
see \cite[11.1.1]{Pi1}. It turns out that Bernstein numbers are $s$-numbers in this sense.
Sometimes these changes of the definition cause problems, e.g., Li and Fang \cite{LiFa}, by investigating the asymptotic behaviour of  $b_n(id: B_{p_1,q_1}^{t+s}(\Omega)\to B_{p_2,q_2}^{s}(\Omega))$ (here $\Omega$ is a bounded Lipschitz domains) used the class $\mathcal{L}_{p,\infty}^{bern}$ overlooking that this class is not an operator ideal.

In the literature, the order of Bernstein numbers was studied in different situations. In the one-dimensional periodic situation, the Bernstein widths of the class $\mathring{W}^t_{p_1}$ in $L_{p_2}$ were calculated by Tsarkov and Maiorov, see \cite[Theorem 12, page 194]{Ti2}. For the class $S^{\alpha}_{p}W$ of periodic functions with bounded mixed derivative the corresponding result has been established by Galeev \cite{Gale1} with the condition $1<p,q<\infty$ and sufficiently high order of smoothness. In the same paper, the order of Bernstein widths in ${L}_{q}$ of the class $S^{\alpha}_{p,\infty}B$ of periodic functions of several variable with bounded mixed difference in the cases of $1<q\leq p<\infty$ and $1<p,q<2$ with high smoothness have been found. Later on, the asymptotic behaviour of Bernstein widths of the class $F^{l,w}_{p}((0,1)^d)$ in the space $L_{q}((0,1)^d)$ were carried out by Kudryavtsev \cite{Ku1}. Recently, estimate of Bernstein widths for classes of convolution functions
with kernels satisfying certain oscillation properties and classes of periodic functions with formal self-adjoint linear differential operators have been established, see \cite{Fen2,Fen1}.

Estimates of the asymptotic behaviour of $s$-numbers, namely approximation, Kolmogorov, Gelfand and Weyl numbers, of embeddings $id: B^t_{p_1,q}((0,1)^d)\rightarrow L_{p_2}((0,1)^d)$ were given by many authors, see Lubitz \cite{Lub}, K\"onig \cite{Koe}, Caetano \cite{Cae1, Cae2, Cae3}, Edmunds and Triebel \cite{ET0,ET1} and Vybiral \cite{Vy2}. Concerning the $s$-numbers of embeddings $ S^t_{p_1,p_1}B((0,1)^d) \rightarrow L_{p_2}((0,1)^d)$, the picture is less complete than in the situation of isotropic Besov spaces. The asymptotic order of approximation and Kolmogorov numbers were studied by Bazarkhanov \cite{Baz4}, Galeev \cite{Gale2} and Romanyuk \cite{Rom7,Rom4,Rom8,Rom1,Rom2,Rom5}. Recently, the order of Weyl numbers has been calculated by Nguyen and Sickel \cite{KiSi}. In this paper we will give a quite satisfactory answer about the asymptotic behaviour of Bernstein numbers for the both cases, isotropic as well as dominating mixed smoothness. 

The paper is organized as follows. Our main results are discussed in Section \ref{sec2}. In Section \ref{sec3} and \ref{sec4}, we will recall definitions of some $s$-numbers and nonlinear widths and discuss their relations to Bernstein numbers. Section \ref{sec5} is devoted to the investigation of the asymptotic behaviour of the Bernstein numbers of embeddings in the case of isotropic Besov spaces. Finally, in Section \ref{sec6}, after estimating the Bernstein numbers of embeddings of certain sequence spaces (which are associated to Besov spaces of dominating mixed smoothness), we shift these results to function spaces.
\subsection*{Notation}

As usual, $\N$ denotes the natural numbers, $\N_0 := \N \cup \{0\}$,
$\zz$ the integers and
$\re$ the real numbers. For a real number $a$ we put $a_+ := \max(a,0)$. By $[a]$ we denote the integer part of $a$.
If $\bar{j} \in \N_0^d$, i.e., if $\bar{j}=(j_1, \ldots \, , j_d)$, $j_\ell \in \N_0$, $\ell=1, \ldots \, , d$, then we put
\[
|\bar{j}|_1 := j_1 + \ldots \, + j_d\, .
\] By $\Omega$ we denote the unit cube in $\R$, i.e., $\Omega:= (0,1)^d$.
If $X$ and $Y$ are two quasi-Banach spaces, then $\mathcal{L}(X,Y)$ denotes the space of all linear bounded operators from $X$ to $Y$.
As usual, the symbol $c $ denotes positive constants 
which depend only on the fixed parameters $t,p,q$ and probably on auxiliary functions, unless otherwise stated; its value may vary from line to line.
Sometimes we will use the symbols ``$ \lesssim $'' 
and ``$ \gtrsim $'' instead of ``$ \le $'' and ``$ \ge $'', respectively. The meaning of $A \lesssim B$ is given by: there exists a constant $c>0$ such that
 $A \le c \,B$. Similarly $\gtrsim$ is defined. The symbol 
$A \asymp B$ will be used as an abbreviation of
$A \lesssim B \lesssim A$.
For a discrete set $\nabla$ the symbol $|\nabla|$ 
denotes the cardinality of this set.
Finally, the symbol $id_{p_1,p_2}^m$ refers to the identity 
\be\label{idlp}
id_{p_1,p_2}^m:~ \ell_{p_1}^m \to \ell_{p_2}^m\, .
\ee
\noindent
{\bf Acknowledgements:}\ I would like to thank Professor Winfried Sickel for many valuable discussions about this work.
\section{The main results}\label{sec2}
Before stating our main results we recall under which conditions the identities $B_{p_1,q}^t(\Omega) \rightarrow L_{p_2}(\Omega)$ and $S_{p_1,p_1}^tB(\Omega) \rightarrow L_{p_2}(\Omega)$ are compact, see \cite[4.3.1]{ET} and \cite[Theorem 3.17]{Vy1}.

\begin{proposition}\label{comp}
{\rm (i)} Let $1\leq p_1,p_2,q\leq \infty$ and $t\in \mathbb{R}$. Then 
$$id: B_{p_1,q}^t(\Omega) \rightarrow L_{p_2}(\Omega)\ \ \text{ is compact if and only if }\ \  \frac{t}{d} > \Big(\frac{1}{p_1}- \frac{1}{p_2} \Big)_+ .$$
{\rm (ii)} Let $1\leq p_1,p_2\leq \infty$ and $t\in \mathbb{R}$. Then 
$$id: S^t_{p_1,p_1}B(\Omega) \rightarrow L_{p_2}(\Omega)\ \  \text{ is compact if and only if }\ \
 t > \Big(\frac{1}{p_1}- \frac{1}{p_2} \Big)_+ .$$
 \end{proposition}
 \noindent
Now we turn to our main results. First we consider embeddings of isotropic Besov spaces. For all admissible sets of parameters the behaviour of the Bernstein numbers is polynomially in $n$.
\begin{theorem}\label{main1}
Let $1\leq p_1,p_2,q\leq \infty$ and $\frac{t}{d}>(\frac{1}{p_1}-\frac{1}{p_2})_+$. Then
 $$b_n(id: B^t_{p_1,q}(\Omega)\rightarrow L_{p_2}(\Omega))\asymp \phi_1(n,t,p_1,p_2)=n^{-\alpha},\ \ n\in \mathbb{N},$$
where
\begin{enumerate}
\item {\makebox[3.2cm][l]{$\alpha=\frac{t}{d}$} if\ \ $ p_2<p_1 \le 2$ or $p_1\leq p_2$};
\item {\makebox[3.2cm][l]{$\alpha=\frac{t}{d} - \frac{1}{p_1} + \frac 12$} if\ \ $p_2 \le 2 \le p_1 ,\ \frac{t}{d}>\frac{1}{p_1}$};
\item {\makebox[3.2cm][l]{$\alpha=\frac{t}{d} + \frac{1}{p_2} - \frac{1}{p_1} $} if\ \ $2 \le p_2 \le p_1,\ \frac{t}{d}> \frac{1/p_2 - 1/p_1}{p_1/2 - 1} $};
\item {\makebox[3.2cm][l]{$\alpha= \frac{tp_1}{2d} $} if\ \ $\big(2 \le p_2 < p_1,\ \frac{t}{d} < \frac{1/p_2 - 1/p_1}{p_1/2 - 1}\big)$ or $\big(1< p_2 \le 2 \le p_1,\ \frac{t}{d}< \frac{1}{p_1}\big)$}.
\end{enumerate}
\end{theorem}

\noindent
In the situation of dominating mixed smoothness, our result reads as follows.
\begin{theorem}\label{main2}
Let $1\leq p_1\leq \infty$, $1<p_2<\infty$ and $t>(\frac{1}{p_1}-\frac{1}{p_2})_+$. Then
$$b_n(id: S_{p_1,p_1}^tB(\Omega)\rightarrow L_{p_2}(\Omega))\asymp \phi_2(n,t,p_1,p_2)=n^{-\alpha}(\log n)^{(d-1)\beta},\ \ n\geq 2,$$
where 
\begin{enumerate}
\item {\makebox[4.5cm][l]{$\alpha=t,\ \beta=0$} if\ \ $ p_1,p_2 \le 2,\ t<\frac{1}{p_1}-\frac{1}{2}$};
\item {\makebox[4.5cm][l]{$\alpha=t,\ \beta=t+\frac{1}{2}-\frac{1}{p_1}$} if\ \ $\big( p_1,p_2 \le 2,\ t>\frac{1}{p_1}-\frac{1}{2}\big)$ or $p_1\leq 2\leq  p_2$};
\item {\makebox[4.5cm][l]{$\alpha=\beta=t - \frac{1}{p_1} + \frac 12$} if\ \ $p_2 \le 2 \le p_1,\ t>\frac{1}{p_1}$};
\item {\makebox[4.5cm][l]{$\alpha=\beta=t + \frac{1}{p_2} - \frac{1}{p_1} $} if\ \ $2 \le p_2 \le p_1,\ t>  \frac{1/p_2 - 1/p_1}{p_1/2 - 1} $};
\item {\makebox[4.5cm][l]{$\alpha= \frac{tp_1}{2},\ \beta=t - \frac{1}{p_1} + \frac 12 $} if\ \ $\big(2 \le p_2 < p_1,\ t < \frac{1/p_2 - 1/p_1}{p_1/2 - 1}\big)$ or $\big(p_2 \le 2 \le p_1,\ t<\frac{1}{p_1}\big).$}
\end{enumerate}
\end{theorem}
\begin{remark}
\rm 
{\rm (i)} We are not aware of any other result concerning the behaviour of Bernstein numbers in the situation of low smoothness. Theorems \ref{main1} and  \ref{main2} gives the final answer about the behaviour of the Bernstein numbers in almost all cases except in limiting cases and the case $2<p_1<p_2<\infty$ of Besov spaces of dominating mixed smoothness. However in this case we can give the estimate from below and above. That is if $2<p_1<p_2<\infty$ then
\beqq
n^{-t}(\log n)^{(d-1)t}\lesssim b_n(id: S_{p_1,p_1}^tB(\Omega)\rightarrow L_{p_2}(\Omega))\lesssim n^{-t}(\log n)^{(d-1)(t-\frac{1}{p_1}+\frac{1}{2})},\ \ n\geq 2.
\eeqq 
{\rm (ii)} Observe that in the case of dominating mixed smoothness the region $1\leq p_1,p_2 \leq 2$ is divided into two regions: low smoothness and high smoothness. \\
{\rm (iii)} For the case $d=1$ the results of the two theorems coincide. This must be the case since $S^t_{p,p}B(0,1)=B^t_{p,p}(0,1)$.
\end{remark}
\begin{remark}\rm It is interesting that our results in Theorem \ref{main1} yield a counterexample for the multiplicativity of Bernstein numbers. Let us consider the commutative diagram

\tikzset{node distance=4cm, auto}
\begin{center}
\begin{tikzpicture}
 \node (H) {$B^{t}_{p_1,q}(\Omega)$};
 \node (L) [right of =H] {$L_2(\Omega)$};
 \node (L2) [right of =H, below of =H, node distance = 2cm ] {$L_\infty(\Omega)$};
 \draw[->] (H) to node {$id_1$} (L);
 \draw[->] (H) to node [swap] {$id$} (L2);
 \draw[->] (L2) to node [swap] {$id_2$} (L);
 \end{tikzpicture}
\end{center}
here $2\leq p_1<\infty$ and $t>\frac{1}{p_1}.$ From 
$$b_{2n}(id_1)\asymp n^{-\frac{t}{d}-\frac{1}{2}+\frac{1}{p_1}},\qquad   b_{n}(id)\asymp n^{-\frac{t}{d}}, $$
see Theorem \ref{main1}, and $b_n(id_2)=(n+1)^{-\frac{1}{2}}$, see \cite{OM} we find
\[
b_{2n} (id_1)
\asymp n^{\frac{1}{p_1}}b_n (id)b_n(id_2)\, , \ \  n\in \mathbb{N} . 
\]
\end{remark}
\begin{remark}\rm It is proved in \cite{Tr70,KiSi} that Kolmogorov, Gelfand and Weyl numbers have the interpolation properties, see Section 3 for definition of these numbers. In this paper we cannot give the complete answer about the interpolation properties of Bernstein numbers. However we can show that Bernstein numbers do not possess interpolation properties concerning with the original spaces. For the basics in interpolation theory we refer to \cite{Tr78}. By $[X_1,X_2]_{\theta}$, $\theta\in (0,1)$, we denote the classical complex method of Calder\'on, here $X_1$ and $X_2$ are Banach spaces. Let $2\leq p_0<p<p_1< \infty$, $1\leq q< \infty$ and $t\in \re$ such that $\frac{1}{p}=\frac{1-\theta}{p_0}+\frac{\theta}{p_1}.$ Then
\beqq
[B^t_{p_0,q}(\Omega),B^t_{p_1,q}(\Omega)]_{\theta}=B^t_{p,q}(\Omega),
\eeqq
see \cite[1.11.8]{Tr06}. However, from Theorem \ref{main1} we have
\beqq
b_{2n-1}(id&&\hspace{-8mm}:B^t_{p,q}(\Omega)\to L_p(\Omega))\\ &\asymp& n^{(\frac{1}{p}-\frac{1}{p_1})\theta}b_{n}^{1-\theta}(id: B^t_{p_0,q}(\Omega)\to L_p(\Omega)) b_{n}^{\theta}(id: B^t_{p_1,q}(\Omega)\to L_p(\Omega)),\  \ n\in \N,
\eeqq
if $t$ large enough.
\end{remark}
We comment on our proof. Concerning the estimate from above, we wish to mention that the standard argument used to estimate $s$-numbers is given by their additivity. 
It consists in a splitting of the identity into a series of finite dimensional operators and reducing the task to the numbers $s_n(id: \ell_{p_1}^m\to \ell_{p_2}^m)$. However, Bernstein numbers fail to be additive so the method does not carry over the present situation. Also the fact that Bernstein numbers are dominated by approximation, Kolmogorov and Gelfand numbers does not help here since it is too rough in general. In this paper we shall use that Weyl numbers and certain nonlinear widths are appropriate to estimate the upper bounds of Bernstein numbers. The estimate from below can be carried out in accordance with the same pattern as in \cite{Lub} and \cite{KiSi}.
\subsection*{Comparison with Weyl numbers}
\begin{definition}
Let $X,Y$ be two Banach spaces and $T\in \mathcal{L}(X,Y)$. The $n$-th Weyl number of $T$ is defined as
$$ x_n(T):=\sup\{a_n(TA):\ A\in \mathcal L(\ell_2,X),\ \|A\|\leq 1\}, $$
where $a_n(TA)$ is the $n$-th approximation number of operator $TA$, see Section \ref{sec3}.
\end{definition}
In general, Bernstein and Weyl numbers are incomparable, see Section \ref{sec3}. The asymptotic behaviour of Weyl numbers in case  of isotropic Besov spaces were investigated in \cite{Cae1,Cae2,Koe,Lub}. In the situation of dominating mixed smoothness we refer to Nguyen and Sickel \cite{KiSi}. To compare Bernstein and Weyl numbers we shall use an $(1/p_1; 1/p_2)$-plane.

\begin{minipage}[t]{7cm}
$$
\begin{tikzpicture}
\fill (0,0) circle (1.5pt);
\draw[->, ](0,0) -- (7,0);
\draw[->, ] (0,0) -- (0,6.8);
\draw (3,0)-- (3,6);
\draw (0,0) -- (3,3);
\draw (0,3) -- (6,3);
\draw (0,6) -- (6,6);
\draw (6,0) -- (6,6);
\node[below] at (0,0) {$0$};
\node [below] at (3,0) {$\frac{1}{2}$};
\node [left] at (0,3) {$\frac{1}{2}$};
\node [left] at (0,6) {$ 1$};
\node [left] at (0,6.6) {$\frac{1}{p_2}$};
\node [below] at (6,0) {$1$};
\draw (3, -0.05) -- (3, 0.05);
\node [below] at (7,0) {$\frac{1}{p_1}$};
\node [] at (4.5,4) {$ b_n \asymp x_n$};
\node [] at (4.5,5) {$I $};
\node [] at (4.5,1) {$\lim\limits_{n \to \infty}\, \frac{b_n}{x_n} = 0 $};
\node [] at (4.5,2) {$II$};
\node [] at (1.5,5) {$V$};
\node [] at (1.5,4) {$b_n\asymp x_n $};
\node [right] at (2,3/2) {$III$};
\node [] at (1.8,0.4) {$\lim\limits_{n \to \infty}\, \frac{b_n}{x_n} = 0 $};
\node [above] at (1,2.2) {$IV$};
\node [] at (1,2) {$b_n\asymp x_n $};
\node [right] at (2.5,-1) {Figure 1};
\end{tikzpicture}
$$
\end{minipage}
\hfill
\begin{minipage}[t]{5.5cm}
{~}\\ 
By $b_n$ and $x_n$ we denote the $n$-th Bernstein number and the $n$-th Weyl number of the embedding $id: X(\Omega)\to L_{p_2}(\Omega)$ respectively; here $X(\Omega)$ is either $S_{p_1,p_1}^tB(\Omega)$ or $B_{p_1,q}^t(\Omega)$. Figure 1 explains the different behaviour of Bernstein and Weyl numbers. In our particular situation Bernstein numbers are dominated by Weyl numbers. There is no difference between the situation of isotropic and non-isotropic smoothness.
\end{minipage}
\section{Bernstein numbers for embeddings of sequence spaces and relations with $s$-number}\label{sec3}
As explained above, Bernstein numbers are not $s$-numbers but they are closely related to $s$-numbers. In this section, we will discuss these relations. First, note that if $X,Y$ are Hilbert spaces and $T\in \mathcal{L}(X,Y)$ we have
\beq
b_n(T)=s_n(T),\ \ n\in\mathbb{N},\label{key1}
\eeq
see \cite[Theorem 11.3.4]{Pi1} or \cite[Proposition 1.6]{Lub}. Now we recall the definitions of some other $s$-numbers, see \cite[Chapter 2]{Pi21}. Let $X$ and $Y$ be two Banach spaces and $T\in \mathcal{L}(X,Y)$. Then
\begin{enumerate}
\item The $n$-th approximation number of $T$ is defined as
$$ a_n(T):=\inf\{\|T-A\|: \ A\in \mathcal L(X,Y),\ \ \text{rank} (A)<n\}. $$
\item The $n$-th Kolmogorov number $T$ is defined as
$$ d_n(T):=\inf\{\|Q_NT \|:\dim(N)<n\},$$
where $N$ is a subspace of $Y$ and $Q_N$ is a canonical surjection from $Y$ onto the quotation space $Y/N$.
\item The $n$-th Gelfand number of $T$ is defined as
$$ c_n(T):=\inf\{\|TJ_M\|,\ \ \text{codim}(M)< n\},$$
where $M$ is a subspace of $X$ and $J_M$ is a canonical injection from $M$ into $X$.
\end{enumerate}
For definitions of Weyl numbers, see Section \ref{sec2}. We have the relations
$$ b_n(T)\leq c_n(T), d_n(T)\leq a_n(T),\ \ n\in \mathbb{N},$$
see \cite{Pie4}. Concerning Weyl and Bernstein numbers, an inequality of type either $b_n(T)\leq cx_n(T)$ or $x_n(T)\leq cb_n(T)$ cannot hold. For instance, consider $b_n(id: \ell_1^n\to \ell_1^n)=1$, $x_n(id: \ell_1^n\to \ell_1^n)=\frac{1}{\sqrt{n}}$ or $b_n(id: \ell_1\to c_0)=\frac{1}{n}$, $x_n(id: \ell_1\to c_0)=1$, see \cite{Pie2} and the references given there. However, the inequality $b_n(T)\leq cx_n(T)$ holds true under some additional restrictions. Therefore we shall need the following lemma of Pietsch \cite{Pie2}.
\begin{lemma}Let $X,Y$ be two Banach spaces and $T\in\mathcal{L}(X,Y)$. Then
\be\label{bern1} b_{2n-1}(T)\leq e\Big(\prod_{k=1}^{n}x_k(T)\Big)^{\frac{1}{n}},\ \ n\in \N.\ee
\end{lemma}
\begin{corollary}
\label{bern-weyl0}
Let $X,Y$ be two Banach spaces, $T \in\mathcal{L}(X,Y)$ and $\alpha,\beta\in \mathbb{R},\ \beta\geq 0$. Assume that $x_n(T)\asymp n^{-\alpha}(\log n)^{\beta}$, $n\geq 2$. Then
$$b_n(T)\lesssim x_n(T) ,$$ 
for all $n\in \mathbb{N}$. Moreover, if $Y$ is a Hilbert space we have $b_n(T)\asymp x_n(T)$.
\end{corollary}
\begin{proof}
The inequality of Pietsch and our assumption $x_n(T)\asymp n^{-\alpha}(\log n)^{\beta} $ yield
$$b_{2n-1}(T)\leq e\Big(\prod_{k=1}^{n}x_k(T)\Big)^{\frac{1}{n}}\lesssim \Big(\prod_{k=2}^{n}k^{-\alpha}(\log k)^{\beta}\Big)^{\frac{1}{n}}\lesssim (\sqrt[n]{n!})^{-\alpha} (\log n)^{\beta} .$$
From the well-known limit 
$$\lim_{n\to \infty}\frac{n}{\sqrt[n]{n!}}=e$$
we obtain
$$b_{2n-1}(T)\lesssim n^{-\alpha}(\log n)^{\beta} .$$
Now, if $Y$ is a Hilbert space and  $A\in \mathcal L(\ell_2,X)$, $\|A\|\leq 1$, property (s3) and (\ref{key1}) yield
$$a_n(TA)= b_n(TA)\leq b_n(T)\|A\|\leq b_n(T).$$
From the definition of Weyl numbers we conclude
$$x_n(T)\leq b_n(T).$$ 
The proof is complete.
\end{proof}
\begin{remark}\rm
The converse inequality of (\ref{bern1}) $x_{2n-1}(T)\leq c\Big(\prod_{k=1}^{n}b_k(T)\Big)^{\frac{1}{n}}$ fails to hold for any constant $c>0$, see \cite{Pie2}.
\end{remark}
We proceed with the relations between Bernstein and Gelfand numbers. The following result was proved by Pietsch \cite{Pie4}.
\begin{lemma}\label{bern-gel}
Let $\dim(X)=\dim(Y)=m$ and $T\in \mathcal{L}(X,Y)$. If $T$ is invertible then
$$b_n(T)c_{m-n+1}(T^{-1})=1,$$
for $n\in \mathbb{N}$, $1\leq n\leq m.$ 
\end{lemma}
The behaviour of the Gelfand numbers of  embeddings $\ell_{p_1}^m\to \ell_{p_2}^m$ is known in literature. Here we collect some results  on $c_n(id: \ell_{p_1}^m\to \ell_{p_2}^m)$; cf. \cite{Glus1,Lub,Vy2}.

\begin{lemma}\label{gelfand}
 Let $n,m\in \mathbb{N}$.
\begin{enumerate}
\item If $1\leq p_1,p_2\leq \infty$, then
$$c_n(id_{p_1,p_2}^{2n})\asymp
\begin{cases}
1 &\text{if }\ 2\leq p_1\leq p_2,\\
n^{\frac{1}{2}-\frac{1}{p_1}}  &\text{if }\  p_1\leq 2\leq p_2,\\
n^{\frac{1}{p_2}-\frac{1}{p_1}} &\text{if }\ p_2\leq p_1\ \text{  or  }\  p_1\leq p_2\leq 2.\\
\end{cases}
$$
\item If $1< p_1$ and $\max(p_1,2)\leq p_2\leq \infty $, then
$$c_{m-n}(id_{p_1,p_2}^{m})\asymp m^{\frac{1}{p_2}-\frac{1}{p_1}},\ \ 1\leq n\leq m^{\frac{2}{p_2}} .$$
\end{enumerate}
\end{lemma}

As a consequence of Lemmas \ref{bern-gel} and \ref{gelfand} we obtain the following.
\begin{corollary}\label{ber}
 Let $n,m\in \mathbb{N}$.
\begin{enumerate}
\item Let $1\leq p_1,p_2\leq \infty$. It holds
\begin{subnumcases}{b_n(id_{p_1,p_2}^{2n}) \gtrsim}
1&\text{ if } $2\leq p_2\leq p_1$,\label{th1} \\
n^{\frac{1}{p_2}-\frac{1}{2}}&\text{ if } $ p_2\leq 2\leq p_1$,\label{th2} \\
n^{\frac{1}{p_2}-\frac{1}{p_1}}&\text{ if } $ p_1\leq p_2 \ \text{ or }\ p_2\leq p_1\leq 2$.\label{th3}
\end{subnumcases}
\item Let $1< p_2$ and $\max(p_2,2)\leq p_1\leq \infty$. It holds
\be\label{th4}
 b_n(id_{p_1,p_2}^m)\gtrsim m^{\frac{1}{p_2}-\frac{1}{p_1}},\ \ 1\leq n\leq m^{\frac{2}{p_1}}.\ee
\end{enumerate}

\end{corollary}

\begin{remark}\rm Pukhov \cite{Pu1} proved that for $n\in \mathbb{N}$ and $1\leq p_1, p_2\leq \infty$ it holds
$$ b_n(id_{p_1,p_2}^{2n})=\dfrac{1}{d_n(id_{p_1',p_2'}^{2n})}, \qquad \ p_1'=\dfrac{p_1}{p_1-1},\ \ p_2'=\dfrac{p_2}{p_2-1}. $$
Combined with the well-known  estimate for the Kolmogorov numbers $d_n(id_{p_1,p_2}^m)$, see \cite{Glus1} or \cite{Vy2}, we also obtain Corollary \ref{ber} (i).
\end{remark}

\section{Nonlinear widths and Bernstein numbers}\label{sec4}
In this section we will recall several definitions of nonlinear widths which have been studied over the last decades. Afterwards we will derive comparisons between them and Bernstein numbers, we refer to \cite{Devo2, Devo1, Dung, Dung2, Dung1}. First, we would like to recall the well-known Alexandroff $n$-width. We start with the definition of a complex, see \cite[Chapter IV]{Alek} or \cite{Devo1}. Let $\{x_1,...,x_N\}$ be a set of points in a normed linear space $X$ and let $\mathcal{T}$ be a collection of subsets of $\mathcal{N}=\{1,...,N\}$ satisfying the following conditions
\begin{enumerate}
\item $T\in \mathcal{T}$ implies that $\{x_j\}_{j\in T}$ are in general position. A set of points $\{x_0,x_1,...,x_n\}$ is called in general position if the system $\{x_1-x_0,...,x_n-x_0\}$ is linearly independent.
\item If $T'\subset T$ and $T\in \mathcal{T}$, then $T'\in \mathcal{T}$.
\item For each $T\in \mathcal{T}$, we denote the simplex associated to $T$ by $\sigma_{T}=\text{conv}\{x_j;j\in T\}$. If $T_1,T_2\in \mathcal{T}$ there exists an $I\in \mathcal{T}$ such that $\sigma_{T_1}\cap \sigma_{T_2}=\sigma_{I}$.
\end{enumerate}
Then $\mathscr{C}=\mathscr{C}(x_1,...,x_n;\mathcal{T})=\cup_{T\in \mathcal{T}}\sigma_{T}$ is called a complex in $X$. The dimension of $\mathscr{C}$ is defined by $\text{dim }\mathscr{C}:=\max\{|T|-1:T\in \mathcal{T}\}$. For a compact subset $K$ of $X$, the Alexandroff nonlinear $n$-width is defined as follows
\begin{equation}\label{lin1}
a_n(K,X)=\inf_{\text{dim } \mathscr{C}=n}\inf_{F\in C(K,\mathscr{C})}\sup_{x\in K}\|x-F(x)|X \|,
\end{equation}
where $C(K,\mathscr{C})$ is the set of all continuous functions from $K$ into $\mathscr{C}$ and the outer infimum is taken over all $n$-dimensional complexes $\mathscr{C}$ in $X$. We also have the dual notion of $n$-co-width
\be\label{lin2}
a^n(K|X)=\inf_{\text{dim } \mathscr{C}=n}\inf_{G\in C(K,\mathscr{C})}\sup_{x\in K}\text{diam}(G^{-1}(G(x))).
\ee
Here as usual 
$$\text{diam}(G)=\sup\{\|x-y|X\|:\ x,y\in G\}\ \ \ \text{and}\ \ \ G^{-1}(A)=\{x:\ G(x)\in A\}.$$ 
Note that $\mathscr{C}$ is not necessarily in $X$. The nonlinear manifold $n$-width $\delta_n(K,X)$ is defined by
\be\label{lin4}
\delta_n(K,X)=\inf_{M,a}\sup_{x\in K}\| x-M(a(x))|X \|,
\ee
where the infimum is taken over all continuous mappings $a$ from $K$ into $\mathbb{R}^n$ and $M$ from $\mathbb{R}^n$ into $X$. For other notions of nonlinear widths, see \cite{Devo2, Devo1, Dung, Dung2, Dung1}. It is obvious that the nonlinear $n$-widths introduced in (\ref{lin1})-(\ref{lin4}) are different. However, they share some common properties and are closely related. The following lemma goes back to \cite{Devo1}, see also \cite{Dung,Dung2,Dung1}.
\begin{lemma}\label{equiv}
For any normed linear space $X$ and any compact set $K\subset X$, we have
$$a_n(K,X)\asymp a^n(K|X)\asymp \delta_n(K,X)$$
for all $n\in \mathbb{N}$.
\end{lemma}
Because of the asymptotic equivalence, in this and following sections, we denote by $w_n$ any of the nonlinear widths defined in (\ref{lin1})-(\ref{lin4}). The following proposition shows that Bernstein numbers and nonlinear widths are tightly related.
\begin{proposition}
\label{bern-width}
Let $K$ be any compact subset of the normed linear space $X$. Then
\be b_n(K,X)\leq \delta_{n}(K,X).\label{b-w1} \ee
for all $n\in \mathbb{N}$.
\end{proposition}
\begin{remark}\rm
{\rm (i)} Inequality (\ref{b-w1}) was proved by DeVore, Howard and Micchelli \cite{Devo2}. There is an other counterpart of the inequality (\ref{b-w1}), proved by Tikhomirov \cite[Theorem 3, p. 190]{Ti2}. He showed that $b_n(K,X)\leq 2a_{n}(K,X) $.\\
{\rm (ii)}
The inequality (\ref{b-w1}) is very useful because in many cases the upper bound of $\delta_{n}(K,X) $ (lower bound of $b_n(K,X)$) can be used for the upper estimate of $b_n(K,X)$ (lower estimate of $\delta_{n}(K,X) $).
\end{remark}
\section{Bernstein numbers of embeddings of isotropic Besov spaces}\label{sec5}
There are several definitions of isotropic Besov spaces which are currently in use, see \cite{Devo3} and \cite[2.3.1]{Tr83}. 
These definitions are equivalent with certain restrictions on the parameters. In this section, we introduce Besov spaces by using the modulus of smoothness and describe a B-spline representation for functions in these spaces. For the following we refer to \cite{Devo3}. $ $

Let $1\leq p,q\leq \infty$ and $f\in L_p(\Omega)$. By $\omega_r(f,s)_p, s>0$, we denote the modulus of smoothness of order $r$ of $f\in L_p(\Omega)$
\be\label{ct1}
\omega_r(f,s)_p:=\sup_{|h|\leq s}\|\Delta_h^r(f,.)|L_p(\Omega(rh))\|,
\ee
where $|h|$ is the Euclidean length of the vector $h\in \R$; $\Delta_h^r$ is the $r$-th order difference with step $h$; and the norm in (\ref{ct1}) is the $L_p$ norm taken on $\Omega(rh)=\{x: x,x+rh\in \Omega\}$. If $ p,q\geq 1 $ and $t>0$, we say that $f\in L_p(\Omega)$ belongs to the Besov space $B^t_{p,q}(\Omega)$ whenever 
$$|f|_{B^{t}_{p,q}(\Omega)}=\Big(\int_{0}^{\infty}(s^{-t}\omega_r(f,s)_p)^q\frac{ds}{s}\Big)^{\frac{1}{q}}$$
is finite. Here $r$ is any integer which is larger than $t$. When $q=\infty$ the usual modification takes place. We define
$$\|f|B_{p,q}^t(\Omega)\|= \|f|L_p(\Omega)\|+|f|_{B^{t}_{p,q}(\Omega)}.$$
For $t>0$, we fix the integer $r$ so that $r>t$. Let $N(x)=N(x,0,\dots,r)$ be the univariate B-spline of degree $r-1$ which has knots at the points $0,1,...,r$. For $d$ dimension we define $\mathcal{N}$ as the tensor product B-spline
$$\mathcal{N}(x):=N(x_1)\cdots N(x_d), \ x=(x_1,...,x_d).$$
The splines $\mathcal{N}$ are nonnegative and have support on the cube $[0,r]^d$. We define the shifted dilates
$$\mathcal{N}_{{j},k}(x)=\mathcal{N}(2^{k}x-{j}),\ {j}\in \mathbb{Z}^d,\ k\in\N_0,\ x\in \mathbb{R}^d.$$
Then the B-splines $\mathcal{N}_{{j},k}$ form a partition of unity, i.e.,
$$\sum_{{j}\in \mathbb{Z}^d}\mathcal{N}_{{j},k}=1,\ \forall x\in\mathbb{R}^d.$$
Denote $\varLambda_k$ by the set of those indices ${j}\in \mathbb{Z}^d$ such that $\mathcal{N}_{{j},k}$ does not vanish identically on $\Omega$. Let $S_{k}$ denote the linear span of the B-splines $\mathcal{N}_{{j},k}$, ${j}\in \varLambda_k$. Then any $f\in S_{k}$ can be represented as
$$f=\sum_{{j}\in\varLambda_k }a_{{j},k}(f)\mathcal{N}_{{j},k},$$
where $a_{{j},k} $ is the  dual functional of $\mathcal{N}_{{j},k}$. It is proved in \cite{Devo3} that for any $f\in B^t_{p,q}(\Omega)$ we have the following atomic decomposition
$$f =\sum_{k=0}^{\infty} \sum_{{j}\in\varLambda_k }a_{{j},k}(f)\mathcal{N}_{{j},k},$$
with convergence in the sense of $L_{p}(\Omega)$. By $S_kL_p(\Omega)$ we denote the space $S_k$ with the norm induced from $L_p(\Omega)$. Let $d_k=\text{dim}S_k$. Then we have $d_k\asymp 2^{kd}$. As a preparation for the proof of Theorem \ref{main1} we recall some helpful results.
\begin{lemma}\label{lem1}Let $k\in \mathbb{N}_0$ and $1\leq p\leq \infty$.
\begin{enumerate}
\item Denote by $J_k: S_kL_p(\Omega)\to \ell_p^{d_k}$ a mapping such that $J_k(f)=(a_{{j},k}(f))_{j\in \varLambda_k}$ for $f\in S_kL_p(\Omega) $. Then
$$ \|J_k: S_kL_p(\Omega)\to \ell_p^{d_k} \|\lesssim 2^{\frac{kd}{p}}\ \ \text{and}\ \ \| J_k^{-1}: \ell_p^{d_k} \to S_kL_p(\Omega)\|
\lesssim 2^{-\frac{kd}{p}},$$
with the constants behind $\lesssim$ depending at most on $r$ and $d$.
\item The projection $T_kf=\sum_{{j}\in\varLambda_k}a_{{j},k}(f)\mathcal{N}_{{j},k}$ satisfies
$$\| T_k: L_p(\Omega)\to S_kL_p(\Omega) \|\leq c,$$
where $c$ independent of $k$.
\item If $r\in \mathbb{N}$, $1\leq q\leq \infty$ and $0<t<\min(r,r-1+\frac{1}{p})$, then for $f\in S_k$, we have
$$\| f|B^t_{p,q}(\Omega) \|\leq c 2^{tk}\| f|S_kL_p(\Omega) \|,$$
with $c$ independent of $f$ and $n$.
\end{enumerate}
\end{lemma}
Parts (i) and (ii) of the above lemma were proved by Lubitz \cite[Lemma 4.4]{Lub}, part (iii) was obtained by DeVore and Popov \cite{Devo3}. Lemma \ref{lem1} allows us to shift the estimate from below of $ b_n(id: B^t_{p_1,q}(\Omega)\rightarrow L_{p_2}(\Omega) )$ to an estimate of $b_n(id_{p_1,p_2}^{m})$. 
\begin{lemma}Let $1\leq p_1,p_2,q\leq \infty$. We have
\be\label{case0}
 b_n(id: B^t_{p_1,q}(\Omega)\rightarrow L_{p_2}(\Omega))\gtrsim m^{-\frac{t}{d}+\frac{1}{p_1}-\frac{1}{p_2}}b_n(id_{p_1,p_2}^{m}),\ee
for $n,m\in \mathbb{N}$ and $1\leq n<m.$
\end{lemma}
\begin{proof}
We consider the following commutative diagram

$$
\begin{CD}
\ell_{p_1}^{d_k} @ >J_k^{-1}>> S_kL_{p_1}(\Omega) @>W_k>>B^t_{p_1,q}(\Omega)\\
@V id_{p_1,p_2}^{d_k} VV @. @V V id V\\
\ell_{p_2}^{d_k} @ <J_k<< S_kL_{p_2}(\Omega)@ <T_k<< L_{p_2}(\Omega)
\end{CD}
$$
Then, the property (s3) yields
$$b_n(id_{p_1,p_2}^{d_k})\leq \|J_k^{-1}\|\,\|W_k\| \,\| T_k\|\, \|J_k\| \, b_n(id).$$
Applying Lemma \ref{lem1}, we obtain
$$b_n(id_{p_1,p_2}^{d_k})\lesssim 2^{tk}2^{-\frac{kd}{p_1}} 2^{\frac{kd}{p_2}}b_n(id ).$$
This implies
$$ b_n(id)\gtrsim 2^{kd(-\frac{t}{d}+\frac{1}{p_1}-\frac{1}{p_2})} b_n(id_{p_1,p_2}^{d_k}).$$
Simple monotonicity arguments allow to switch from $m=2^{kd}$ to general $m\in \mathbb{N}$.
\end{proof}
Note that the inequality (\ref{case0}) also holds for any $s$-number since the only property  we had used is $(s3)$. We shall apply Proposition \ref{bern-width} and Corollary \ref{bern-weyl0} to obtain estimates from above of Bernstein numbers of the embedding $id: B^t_{p_1,q}(\Omega)\rightarrow L_{p_2}(\Omega)$. For that reason, let us recall the behaviour of the nonlinear widths of the unit ball $U(B^t_{p_1,q}(\Omega))$ in $L_{p_2}(\Omega)$, see \cite{Devo1,Dung2}.
\begin{proposition}\label{width1}
Let $1\leq p_1 ,p_2, q \leq \infty$ and $\frac{t}{d}>(\frac{1}{p_1}-\frac{1}{p_2})_+$. Then 
 $$w_n(U(B^t_{p_1,q}(\Omega)), L_{p_2}(\Omega))\asymp n^{-\frac{t}{d}},\ n\in\mathbb{N}.$$
 \end{proposition}
Concerning the behaviour of Weyl numbers of the embedding $B^t_{p_1,q}(\Omega)\rightarrow L_{p_2}(\Omega)$ we refer to \cite{Cae1,Cae2,Koe,Lub}.
\begin{proposition}\label{weylis}

Let $1\leq p_1,p_2,q\leq \infty$ and $\frac{t}{d}>(\frac{1}{p_1}-\frac{1}{p_2})_+$. Then
 $$x_n(id: B^t_{p_1,q}(\Omega)\rightarrow L_{p_2}(\Omega))\asymp \varphi_1(n,t,p_1,p_2)=n^{-\alpha},\ \ n\in \mathbb{N},$$
 where 
\begin{enumerate}
\item \quad{\makebox[3.5cm][l]{$\alpha=\frac{t}{d}$} if\ \ $ p_1,p_2 \le 2$};
\item \quad{\makebox[3.5cm][l]{$\alpha=\frac{t}{d} + \frac{1}{p_2} - \frac{1}{2}$} if\ \ $p_1 \le 2 \le p_2 $};
\item \quad{\makebox[3.5cm][l]{$\alpha=\frac{t}{d} - \frac{1}{p_1} + \frac 12$} if\ \ $p_2 \le 2 \le p_1,\ \frac{t}{d}>\frac{1}{p_1}$};
\item \quad{\makebox[3.5cm][l]{$\alpha=\frac{t}{d} + \frac{1}{p_2} - \frac{1}{p_1} $} if\ \ $\big(2 \le p_2 \le p_1,\ \frac{t}{d}> \frac{1/p_2 - 1/p_1}{p_1/2 - 1}\big) $ or $ 2 \le p_1 \le p_2$};
\item \quad{\makebox[3.5cm][l]{$\alpha= \frac{tp_1}{2d} $} if\ \ $\big(2 \le p_2 < p_1,\ \frac{t}{d} < \frac{1/p_2 - 1/p_1}{p_1/2 - 1}\big)$ or $ \big(p_2 \le 2 \le p_1,\ \frac{t}{d}< \frac{1}{p_1}\big)$}.
\end{enumerate}
\end{proposition}
\noindent
Now we ready to prove the Theorem \ref{main1}.\\
\noindent
{\bf Proof of Theorem \ref{main1}.} From (\ref{case0}), choosing $n=\big[\frac{m}{2}\big]$ and taking into account \eqref{th3}, we obtain
$$b_n(id: B^t_{p_1,q}(\Omega)\rightarrow L_{p_2}(\Omega))\gtrsim n^{-\frac{t}{d}+\frac{1}{p_1}-\frac{1}{p_2}}n^{\frac{1}{p_2}-\frac{1}{p_1}} =n^{-\frac{t}{d}}.$$
This proves the estimate from below in (i). We prove the lower bound in (iv). This time we choose $n=\big[m^{\frac{2}{p_1}}\big]$. In a view of \eqref{case0} and \eqref{th4} we have 
$$b_n(id: B^t_{p_1,q}(\Omega)\rightarrow L_{p_2}(\Omega))\gtrsim m^{-\frac{t}{d}+\frac{1}{p_1}-\frac{1}{p_2}}m^{\frac{1}{p_2}-\frac{1}{p_1}} =m^{-\frac{t}{d}}\asymp n^{-\frac{tp_1}{2d}}.$$
For the lower bound in the other cases we combine the inequality (\ref{case0}) with inequalities (\ref{th1}) or (\ref{th2}) and choose $n=\big[\frac{m}{2}\big]$. Concerning the estimate from above we shall use 
$$b_n\lesssim \min(x_n, \delta_n),\ \ n\in \mathbb{N},$$
where we take into account of the polynomial behaviour of the Weyl numbers of the embedding $B^t_{p_1,q}(\Omega)\rightarrow L_{p_2}(\Omega) $ and  Corollary \ref{bern-weyl0}, Propositions \ref{bern-width}, \ref{width1} and \ref{weylis}. The proof is complete. \qed 
\vskip 2mm
\noindent
Let $t\in \mathbb{R}$ and $ 1<p<\infty$. By $H^t_p(\mathbb{R}^d)$ we denote the fractional order Sobolev spaces:
$$H^t_{p}(\mathbb{R}^d)=\big\{f\in \mathcal{S}'(\mathbb{R}^d):\|f|H^t_{p}(\mathbb{R}^d)\|=\| \mathcal{F}^{-1}\big[(1+|x|^2)^{\frac{t}{2}}\mathcal{F}f\big]| L_{p}(\mathbb{R}^d)\|<\infty\big\}.$$
Here $\mathcal{F}g$ and $\mathcal{F}^{-1}g$ denote the  Fourier transform and its inverse transform of  $g\in \mathcal{S}'(\mathbb{R}^d)$. Note that if $t\in \N$ then $H^t_{p}(\mathbb{R}^d)=W^t_p(\R) $ in the sense of equivalent norms, where $W^t_p(\R)$ is the classical Sobolev space. As the consequence of Theorem \ref{main1} and the chain of continuous embeddings 
$$B_{p_1,1}^t(\Omega)\rightarrow H_{p_1}^t(\Omega)\rightarrow B_{p_1,\infty}^t(\Omega), $$
see \cite[2.3.2, 2.3.5]{Tr83}, we have the following corollary.
\begin{corollary}
Let $1< p_1< \infty$, $1\leq p_2\leq \infty$ and $\frac{t}{d}>(\frac{1}{p_1}-\frac{1}{p_2})_+$. Then
 $$b_n(id: H^t_{p_1}(\Omega)\rightarrow L_{p_2}(\Omega))\asymp \phi_1(n,t,p_1,p_2),  \ \ n\in \mathbb{N},$$
where $\phi_1(nt,p_1,p_2)$ is given in theorem \ref{main1}. 
\end{corollary}
\begin{remark}
\rm 
{\rm (i)} The lower bound in the cases (i) $(p_1\leq p_2)$ and (iii) of Theorem \ref{main1} can be found in \cite{DaNoSi}. Dahlke, Novak and Sickel \cite[consequence of Lemma 5]{DaNoSi} showed that if $1\leq p_1,q,p_2\leq \infty$ and $\frac{t}{d}>(\frac{1}{p_1}-\frac{1}{p_2})_+$ it holds
$$b_n(id: B^t_{p_1,q}(\Omega)\to L_{p_2}(\Omega))\gtrsim 
\begin{cases}
n^{-\frac{t}{d}}  &\text{if }\  p_1\leq p_2,\\
n^{-\frac{t}{d}+\frac{1}{p_1}-\frac{1}{p_2}} &\text{if }\ p_2\leq p_1,
\end{cases}
$$
for all $n\in \mathbb{N}$.\\
{\rm (ii)} In the one-dimensional periodic situation the asymptotic behaviour of Bernstein numbers were obtained by Tsarkov and Maiorov, see \cite[Theorem 12, page 194]{Ti2}. Let $1\leq p_1\leq \infty$ and $t>0$. By $\mathring{W}^t_{p_1}(\mathbb{T})$ we denote the Sobolev spaces
$$\mathring{W}^t_{p_1}(\mathbb{T})=\Big\{f(.)|\ D^tf(.)\in L_{p}(\mathbb{T})\ \text{ and } \int_{\mathbb{T}}f(x)dx=0 \Big\},$$ where $D^tf(.)$ is the Weyl fractional derivative of order $t$ of $f$. We have
$$b_n(id: \mathring{W}^t_{p_1}(\mathbb{T})\to L_{p_2}(\mathbb{T}))\asymp 
\begin{cases}
n^{-t} & \text{if}\ 1<p_1\leq p_2<\infty\ \text{or } 1<p_2\leq p_1\leq 2,\ t>0;\\
n^{-t+\frac{1}{p_1}-\frac{1}{p_2}} & \text{if}\ 2\leq p_2\leq p_1<\infty \text{ and }t>\frac{1}{p_1};\\
n^{-t+\frac{1}{p_1}-\frac{1}{2}}& \text{if}\ 1<p_2\leq 2\leq p_1<\infty \text{ and }t>\frac{1}{p_1},\\
\end{cases}
$$
for all $n\in \mathbb{N}$. This result should be compared with the Theorem \ref{main1} for $d=1$. Observe that the condition $t>\frac{1}{p_1}$ in the case $2\leq p_2\leq p_1<\infty $ is not sharp.
\end{remark}
\section{Bernstein numbers of dominating mixed Besov spaces }\label{sec6}
\subsection{Besov-Lizorkin-Triebel Spaces of Dominating Mixed Smoothness}
Besov-Lizorkin-Triebel spaces of dominating mixed smoothness have been studied systematically by many authors. For a definition of these spaces in Fourier-analytic terms we refer to \cite[Section 2.2]{ST}. Atomic and wavelet decomposition characterizations of spaces of dominating mixed smoothness have been given in \cite{Baz1, Baz2, Baz3,Vy1}. Let $\varphi \in C_0^\infty (\re)$
be a function such that $\varphi (t)=1$ in an open set containing
the origin. Then by means of
\be\label{unity}
\varphi_0 (t) = \varphi (t)\, , \qquad \varphi_j(t)=
\varphi(2^{-j}t)-\varphi(2^{-j+1}t)\, , \qquad t\in \re\, , \quad j
\in \N\, ,
\ee
we get a smooth dyadic decomposition of unity, i.e.,
\[
\sum_{j=0}^\infty \varphi_j(t)= 1 \qquad \mbox{for all}\quad t \in \re\, ,
\]
and $\supp \varphi_j$, $j\geq 1$, is contained in the dyadic annulus $\{t\in \re: \quad a \, 2^j \le |t| \le b \, 2^j \}$
with $0 < a < b < \infty$ independent of $j \in \N$.
By means of
\be\label{unityd}
\varphi_{\bar{j}} := \varphi_{{j_1}}\otimes \ldots \otimes
\varphi_{{j_d}} \, , \qquad \bar{j}=(j_1, \ldots\, , j_d)\in
\N_0^d\, ,
\ee
we obtain a smooth decomposition of unity on $\R$.

\begin{definition} Let $0<p,q \le \infty$ and $t \in\re$.
\\
{\rm (i)}
The Besov space of dominating mixed smoothness $ S^{t}_{p,q}B(\re^d)$ is the
collection of all tempered distributions $f \in \mathcal{S}'(\R)$
such that
\[
 \|\, f \, |S^{t}_{p,q}B(\R)\| :=
\Big(\sum\limits_{\bar{j}\in \N_0^d} 2^{|\bar{j}|_1 t q}\, \|\, \cfi[\varphi_{\bar{j}}\, \cf f](\, \cdot \, )
|L_p(\re^d)\|^q\Big)^{1/q}
\]
is finite. \\
{\rm (ii)} Let $0 < p< \infty$.
The Lizorkin-Triebel space of dominating mixed smoothness $ S^{t}_{p,q}F(\re^d)$ is the
collection of all tempered distributions $f \in \mathcal{S}'(\R)$
such that
\[
 \|\, f \, |S^{t}_{p,q}F(\R)\| :=
\Big\| \Big(\sum\limits_{\bar{j}\in \N_0^d} 2^{|\bar{j}|_1 t q}\, |\, \cfi[\varphi_{\bar{j}}\, \cf f](\, \cdot \, )|^q \Big)^{1/q} \Big|L_p(\re^d)\Big\|
\]
is finite. 
\end{definition}

Next we will describe the wavelet decomposition for Besov-Lizorkin-Triebel spaces of dominating mixed smoothness. We recall a few results from \cite{Vy1}. Let $\bar{\nu} =(\nu_1, \ldots \, ,\nu_d) \in \mathbb{N}_0^d$ and $\bar{m}= (m_1, \ldots \, , m_d)\in \mathbb{Z}^d$. 
Then we put $2^{-\bar{\nu}}\bar{m}=(2^{-\nu_1}m_1,...,2^{-\nu_d}m_d)$ and
\[
Q_{\bar{\nu},\bar{m}} := \Big\{x\in \R: \quad 2^{-\nu_\ell} \, m_\ell < x_\ell < 2^{-\nu_\ell}\, (m_\ell+1)\, , \: \ell = 1, \, \ldots \, , d\Big\} \, .
\]
By $\chi_{{\bar{\nu},\bar{m}}}(x)$ we denote the characteristic function of $Q_{\bar{\nu},\bar{m}}$. 

\begin{definition}\label{sequence1}
If $0<p,q\leq \infty$, $t\in \mathbb{R}$ and
$\lambda:=\lbrace \lambda_{\bar{\nu},\bar{m}}\in\mathbb{C}:\bar{\nu}\in \mathbb{N}_0^d,\ \bar{m}\in \mathbb{Z}^d \rbrace$,
then we define
$$s_{p,q}^t b := \Big\lbrace\lambda: \| \lambda|s_{p,q}^t b\| =
\Big(\sum_{\bar{\nu}\in \mathbb{N}_0^d}2^{|\bar{\nu}|_1 (t-\frac{1}{p})q}\big(\sum_{\bar{m}\in \mathbb{Z}^d}|\lambda_{\bar{\nu},\bar{m}} 
|^p\big)^{\frac{q}{p}}\Big)^{\frac{1}{q}}<\infty \Big\rbrace$$
and, if $p< \infty$, 
$$s_{p,q}^tf=\Big\lbrace\lambda: \| \lambda|s_{p,q}^tf\| =
\Big\|\Big(\sum_{\bar{\nu}\in \mathbb{N}_0^d}\sum_{\bar{m}\in \mathbb{Z}^d}|2^{|\bar{\nu}|_1 t}
\lambda_{\bar{\nu},\bar{m}}\chi_{\bar{\nu},\bar{m}}(.) 
|^q\Big)^{\frac{1}{q}}\Big|L_p(\mathbb{R}^d)\Big\|<\infty \Big\rbrace$$
with the usual modification for $p$ or/and q equal to $\infty$.
\end{definition}

Let $N \in \N$. Then there exists $\psi_0, \psi_1 \in C^N(\re) $, compactly supported, 
\[
\int_{-\infty}^\infty t^m \, \psi_1 (t)\, dt =0\, , \qquad m=0,1,\ldots \, , N\, , 
\]
such that
$\{ 2^{j/2}\, \psi_{j,m}: \quad j \in \N_0, \: m \in \zz\}$, where
\[
\psi_{j,m} (t):= \left\{ \begin{array}{lll}
\psi_0 (t-m) & \qquad & \mbox{if}\quad j=0, \: m \in \zz\, , 
\\
\sqrt{1/2}\, \psi_1 (2^{j-1}t-m) & \qquad & \mbox{if}\quad j\in \N\, , \: m \in \zz\, , 
      \end{array} \right.
\]
is an orthonormal basis in $L_2 (\re)$. 
We put
\[
\Psi_{\bar{\nu}, \bar{m}} (x) := \prod_{\ell=1}^d \psi_{\nu_\ell, m_\ell} (x_\ell)\, . 
\]
Then 
\[
\Psi_{\bar{\nu}, \bar{m}}\, , \qquad \bar{\nu} \in \N_0^d, \, \bar{m} \in \Z\, ,
\]
is a tensor product wavelet basis of $L_2 (\R)$. Vybiral \cite{Vy1} has proved the following.

\begin{lemma}\label{wavelet}
Let $0< p,q \le\infty$ and $t\in \re$. 
\\
{\rm (i)}
There exists $N=N(t,p) \in \N$ s.t. the mapping 
\be\label{wave}
{\mathcal W}: \quad f \mapsto (2^{|\bar{\nu}|_1}\langle f, \Psi_{\bar{\nu},\bar{m}} \rangle)_{\bar{\nu} \in \N_0^d\, , \, \bar{m} \in \Z} 
\ee
is an isomorphism of $S^t_{p,q}B(\R)$ onto $s^t_{p,q}b$.
\\
{\rm (ii)} Let $p <\infty$. Then
there exists $N=N(t,p,q) \in N$ s.t. the mapping ${\mathcal W}$
is an isomorphism of $S^t_{p,q}F(\R)$ onto $s^t_{p,q}f$.
\end{lemma}
We proceed by defining Besov and Lizorkin-Triebel spaces of dominating mixed smoothness on $\Omega$ by restriction.

\begin{definition}Let $0<p,q \le \infty$ and $t \in\re$.
 \begin{description}
  \item(i) The space $S^{t}_{p,q}B(\Omega)$ is the collection of all $f\in
  D'(\Omega)$ such that there exists a distribution $g\in
   S^{t}_{p,q}B(\R)$ satisfying $f = g|_{\Omega}$ and is endowed with the norm
   $$
    \|f|S^{t}_{p,q}B(\Omega)\| = \inf\{\|g|S^{t}_{p,q}B(\R)\|~:~g|_{\Omega} =
    f\}\,.
   $$
   \item(ii) For $0<p<\infty$, the space $S^{t}_{p,q}F(\Omega)$ is the collection of all $f\in
     D'(\Omega)$ such that there exists a distribution $g\in 
     S^{t}_{p,q}B(\R)$ satisfying $f = g|_{\Omega}$ and is endowed with the norm
   $$
    \|f|S^{t}_{p,q}F(\Omega)\| = \inf\{\|g|S^{t}_{p,q}F(\R)\|~:~g|_{\Omega} =
    f\}\,.
   $$
 \end{description}
\end{definition}

We put
\be\label{ws-40}
A_{\bar{\nu}}^{\Omega}:= 
\Big\{\bar{m}\in \mathbb{Z}^d: \quad \supp \Psi_{\bar{\nu},\bar{m}} \cap\Omega \neq \emptyset\Big\}\, ,\qquad \bar{\nu}\in\mathbb{N}_0^d\, .
\ee
For given $ f \, \in S_{p,q}^t A(\Omega)$, $A\in \{B,F\}$, let $\ce f$ be an element of in $ S_{p,q}^t A(\R)$ s.t. 
\[
\| \, \ce f \, | S_{p,q}^t A(\R)\| \le 2 \, \| \, f \, | S_{p,q}^t A(\Omega)\|
\qquad \mbox{and} \qquad (\ce f)_{|_\Omega} = f \, .
\]
We define
\[
g:= \sum_{\bar{\nu} \in \N_0^d} \sum_{\bar{m} \in A_{\bar{\nu}}^{\Omega}} 2^{|\bar{\nu}|_1} \, \langle \ce f, \Psi_{\bar{\nu},\bar{m}} \rangle \, \Psi_{\bar{\nu}, \bar{m}}\, .
\]
Then it follows that $g \in S_{p,q}^t A(\R)$, $g_{|_\Omega} = f $, 
\[
\supp g \subset \{x \in \R: ~ \max_{j=1,...,d}|x_j| \le c_1\}\]
and 
\[
\| \, g \, | S_{p,q}^t A(\R)\| \le c_2 \, \| \, f \, | S_{p,q}^t A(\Omega)\|\, .
\]
Here $c_1,c_2$ are independent of $f$. For that reason, the sequence spaces associated with $\Omega$ are defined by


\begin{definition}
If $0<p\leq \infty$, $0< q\leq \infty$, $t\in \mathbb{R}$ and $ \lambda=\lbrace \lambda_{\bar{\nu},\bar{m}}\in\mathbb{C}:\bar{\nu}\in \mathbb{N}_0^d,\ \bar{m}\in A_{\bar{\nu}}^{\Omega} \rbrace$ then we define
$$s_{p,q}^{t,\Omega}b=\Big\lbrace \lambda: \| \lambda|s_{p,q}^{t,\Omega}b\| =\Big(\sum_{\bar{\nu}\in \mathbb{N}_0^d}2^{|\bar{\nu}|_1( t-\frac{1}{p})q}\big(\sum_{\bar{m}\in A_{\bar{\nu}}^{\Omega}}|\lambda_{\bar{\nu},\bar{m}} |^p\big)^{\frac{q}{p}}\Big)^{\frac{1}{q}}<\infty \Big\rbrace$$
and, if $p< \infty$, 
$$s_{p,q}^{t,\Omega}f=\Big\lbrace \lambda: \| \lambda|s_{p,q}^{t,\Omega}f\| =\Big\|\Big(\sum_{\bar{\nu}\in \mathbb{N}_0^d}\sum_{\bar{m}\in A_{\bar{\nu}}^{\Omega}}|2^{|\bar{\nu}|_1t}\lambda_{\bar{\nu},\bar{m}}\chi_{\bar{\nu},\bar{m}}(.) |^q\Big)^{\frac{1}{q}}\Big|L_p(\mathbb{R}^d)\Big\|<\infty \Big\rbrace.$$
\end{definition}
In addition we need the corresponding building blocks.
\begin{definition}
If $0<p\leq \infty$, $0< q\leq \infty$, $t\in \mathbb{R}$ and $\mu\in \mathbb{N}_0$ 
$$ \lambda=\lbrace \lambda_{\bar{\nu},\bar{m}}\in\mathbb{C}:\bar{\nu}\in \mathbb{N}_0^d,\ |\bar{\nu}|_1=\mu,\ \bar{k}\in  A_{\bar{\nu}}^{\Omega} \rbrace$$
then we define
$$(s_{p,q}^{t,\Omega}b)_{\mu}=\Big\lbrace \lambda: \| \lambda|(s_{p,q}^{t,\Omega}b)_{\mu}\| =\Big(\sum_{|\bar{\nu}|_1=\mu}2^{|\bar{\nu}|_1( t-\frac{1}{p})q}\big(\sum_{\bar{m}\in A_{\bar{\nu}}^{\Omega}}|\lambda_{\bar{\nu},\bar{m}} |^p\big)^{\frac{q}{p}}\Big)^{\frac{1}{q}}<\infty \Big\rbrace$$
and, if $p< \infty$, 
$$(s_{p,q}^{t,\Omega}f)_{\mu}=\Big\lbrace \lambda: \| \lambda|(s_{p,q}^{t,\Omega}f)_{\mu}\| =\Big\|\Big(\sum_{|\bar{\nu}|_1=\mu}\sum_{\bar{m}\in A_{\bar{\nu}}^{\Omega}}|2^{|\bar{\nu}|_1t}\lambda_{\bar{\nu},\bar{m}}\chi_{\bar{\nu},\bar{m}}(.) |^q\Big)^{\frac{1}{q}}\Big|L_p(\mathbb{R}^d)\Big\|<\infty \Big\rbrace.$$
\end{definition}
We recall some helpful results, see \cite{KiSi} and the references given there.
\begin{lemma}\label{ba1}
\begin{enumerate}
\item We have
\[
 |A_{\bar{\nu}}^{\Omega}| \asymp 2^{|\bar{\nu}|_1},\qquad D_{\mu}:= \sum_{|\bar{\nu}|_1=\mu} |A_{\bar{\nu}}^{\Omega}| \asymp \mu^{d-1}2^{\mu} 
\]
with equivalence constants independent of $\bar{\nu}\in \mathbb{N}_0^d$ and $\mu\in \mathbb{N}_0$. 
\item Let $0<p\leq \infty$ and $t\in \mathbb{R}$. Then
$$
(s_{p,p}^{t,\Omega}f)_{\mu}=(s_{p,p}^{t,\Omega}b)_{\mu} =2^{\mu(t-\frac{1}{p})}\ell_p^{D_{\mu}},\ \ \mu\in \mathbb{N}_0,$$
$$ s_{p,p}^{t,\Omega}f=s_{p,p}^{t,\Omega}b.$$
\end{enumerate}
\end{lemma}
\begin{lemma}\label{ba2-1}Let $0<p_1,q_1,p_2,q_2\leq \infty$, $t\in \mathbb{R}$ and $a\in \{b,f\}$. Then
\be
\| \, id_{\mu}^* \, : (s^{t,\Omega}_{p_1,q_1}a)_\mu \to (s^{0,\Omega}_{p_2,q_2}a)_\mu\|
\lesssim 2^{\mu\big(-t+(\frac{1}{p_1}-\frac{1}{p_2})_+\big)}\mu^{(d-1)(\frac{1}{q_2} - \frac{1}{q_1})_+},
\ee
with constant behind $\lesssim$ independent of $\mu\in \mathbb{N}_0$. 
\end{lemma}
\begin{lemma}\label{ba3}
Let $0<p_1<p_2<\infty$, $0<q_1,q_2\leq \infty$, $t\in \mathbb{R}$. Then
\be \|id_{\mu}^*: (s_{p_1,q_1}^{t,\Omega}f)_{\mu}\to (s_{p_2,q_2}^{0,\Omega}f)_{\mu} \| \lesssim 2^{\mu(-t+\frac{1}{p_1}-\frac{1}{p_2})}
\ee
with constant behind $\lesssim$ independent of $\mu\in \mathbb{N}_0$. 
\end{lemma}
\subsection{Upper bounds of nonlinear widths of sequence spaces related to spaces of dominating mixed smoothness}
In this subsection we shall deal with upper bounds of the nonlinear widths of the unit ball $U(s^{t,\Omega}_{p_1,p_1}b)$ in the space $s^{0,\Omega}_{p_2,2}f$ with the condition $t>(\frac{1}{p_1}-\frac{1}{\delta_2})_+$, where $\delta_2=\max(p_2,2)$. For this we define the embedding
$$id^*: s^{t,\Omega}_{p_1, p_1}b \rightarrow s^{0,\Omega}_{p_2, 2}f .$$
A few more notation is needed. For $\mu\in \N_0$ we define
$$\nabla_{\mu} =\{(\bar{\nu},\bar{m})\in \mathbb{N}_0^d\times \mathbb{Z}^d: |\bar{\nu}|_1=\mu,\ \bar{m}\in A_{\bar{\nu}}^{\Omega}\}.$$
Note that $|\nabla_{\mu} |=D_{\mu}\asymp 2^{\mu}\mu^{d-1}$. Furthermore, for $J\in \N$ we denote
$$S_J\lambda=\sum_{\mu=0}^{J}\sum_{(\bar{\nu},\bar{m})\in\nabla_{\mu}}\lambda_{\bar{\nu},\bar{m}}e^{\bar{\nu},\bar{m}},$$
where $\{e^{\bar{\nu},\bar{m}}: \bar{\nu}\in \mathbb{N}_0^d,\ \bar{m}\in A_{\bar{\nu}}^{\Omega}\}$ is the canonical orthonormal basis of 
$s^{0,\Omega}_{2,2}b$ and $ \lambda_{\bar{\nu},\bar{m}}=\langle \lambda, e^{\bar{\nu},\bar{m}} \rangle$.
\begin{lemma}\label{non1-2}
Let  $J\in \mathbb{N}$.
\begin{enumerate} 
\item If $0 < p_1 < p_2 < \infty$ and $t> \frac{1}{p_1} - \frac{1}{p_2}$, we have 
\be\label{bdt3}
\|\, id^* - S_J\, | s^{t,\Omega}_{p_1, p_1}b \rightarrow s^{0,\Omega}_{p_2, 2}f\| 
 \lesssim
2^{J(-t+\frac{1}{p_1}-\frac{1}{p_2})},
\ee
with constant behind $\lesssim$ independent of $J$. 
\item If $0<p_2\leq p_1\leq \infty$ and $t>0$, we have
\be \label{bdt4}
\| id^* - S_J | s^{t,\Omega}_{p_1, p_1}b \rightarrow s^{0,\Omega}_{p_2, 2}f\| 
 \lesssim
2^{-Jt}J^{(d-1)(\frac{1}{2}-\frac{1}{p_1})_+},
\ee
with constant behind $\lesssim$ independent of $J$. 
\end{enumerate}
\end{lemma}
\begin{proof} Let $J\in \mathbb{N}$. Obviously,
$$\|\, id^*-S_J\, | s^{t,\Omega}_{p_1, p_1}b \rightarrow s^{0,\Omega}_{p_2, 2}f\| \lesssim \sum_{\mu=J+1}^{\infty} \|id_{\mu}^*: (s^{t,\Omega}_{p_1, p_1}b)_{\mu} \to (s^{0,\Omega}_{p_2, 2}f)_{\mu}\|.$$
From Lemma \ref{ba3} and $(s^{t,\Omega}_{p_1, p_1}b )_{\mu}=(s^{t,\Omega}_{p_1, p_1}f )_{\mu}$ we obtain
$$\|\, id^*-S_J\, | s^{t,\Omega}_{p_1, p_1}b \rightarrow s^{0,\Omega}_{p_2, 2}f\| \lesssim \sum_{\mu=J+1}^{\infty} 2^{\mu(-t+\frac{1}{p_1}-\frac{1}{p_2})} \lesssim 2^{J(-t+\frac{1}{p_1}-\frac{1}{p_2})}.$$
This proves \eqref{bdt3}. By making use of Lemma \ref{ba2-1} and a similar argument as above we derive the inequality (\ref{bdt4}) as well.
\end{proof}
Employing $\text{rank} S_J\asymp 2^J\, J^{d-1}$ and the monotonicity of the approximation numbers Lemma \ref{non1-2} yields the following results.
\begin{corollary}\label{appro}
Let $t>(\frac{1}{p_1} - \frac{1}{p_2})_+$.
\begin{enumerate}
\item If $0 < p_1 < p_2 < \infty$, we have
$$a_n(id^*: s^{t,\Omega}_{p_1, p_1}b \rightarrow s^{0,\Omega}_{p_2, 2}f )\lesssim n^{-t+ \frac{1}{p_1} - \frac{1}{p_2}}\, 
(\log n)^{(d-1)(t - \frac{1}{p_1} + \frac{1}{p_2})} ,$$
for all $n\geq 2$.
\item If $0 < p_2\leq \delta_2 \leq p_1 \leq \infty$, we have
$$a_n(id^*: s^{t,\Omega}_{p_1, p_1}b \rightarrow s^{0,\Omega}_{p_2, 2}f )\lesssim n^{-t}\, (\log n)^{(d-1)(t - \frac{1}{p_1} + \frac{1}{2})},$$
for all $n\geq 2$.
\end{enumerate}
\end{corollary}

\begin{lemma}
\label{up-width}
 Let $p_1<\delta_2$ and $t>\frac{1}{p_1}-\frac{1}{\delta_2}$. Then for each $J\in \mathbb{N}$, there is an positive integer $K=K(J)$ such that 
$$ \| id^*- S_K |s^{t,\Omega}_{p_1, p_1}b \rightarrow s^{0,\Omega}_{p_2, 2}f\|\lesssim  2^{-Jt}J^{(d-1)(\frac{1}{2}-\frac{1}{p_1})},$$
with constant independent of $J$ and $K$. 
\end{lemma}
\begin{proof}
Let $J\in \mathbb{N}$. For $L\in \mathbb{N}$, from Lem \ref{non1-2} we conclude
 $$\| id^*- S_L |s^{t,\Omega}_{p_1, p_1}b \rightarrow s^{0,\Omega}_{p_2, 2}f\|\lesssim 2^{L(-t+(\frac{1}{p_1}-\frac{1}{p_2})_+)}L^{(d-1)\alpha}$$
for some $\alpha=\alpha(t,p_1)$. Because of $-t+(\frac{1}{p_1}-\frac{1}{p_2})_+<0 $ we can choose $L$ large enough such that 
$$2^{L(-t+(\frac{1}{p_1}-\frac{1}{p_2})_+)}L^{(d-1)\alpha}\leq 2^{-Jt}J^{(d-1)(\frac{1}{2}-\frac{1}{p_1})}. $$
Then we obtain
\beq
\| id^*- S_L |s^{t,\Omega}_{p_1, p_1}b \rightarrow s^{0,\Omega}_{p_2, 2}f\|\lesssim 2^{-Jt}J^{(d-1)(\frac{1}{2}-\frac{1}{p_1})}.
\label{tem}\eeq
Now we choose $K$ as the smallest number $L$ satisfying the inequality (\ref{tem}). The proof is complete.
\end{proof}
\begin{proposition}\label{non2}
Let $p_1<\delta_2$ and $t>\frac{1}{p_1}-\frac{1}{\delta_2}$. Then for all $\lambda\in s^{t,\Omega}_{p_1, p_1}b$, $\|\lambda|s^{t,\Omega}_{p_1, p_1}b \|\leq 1$ and $J\in \mathbb{N}$ there exists an approximation of the form
$$S_K^*\lambda=\sum_{\mu=0}^{K}\sum_{(\bar{\nu},\bar{m})\in \nabla_{\mu}}\lambda_{\bar{\nu},\bar{m}}^*e^{\bar{\nu},\bar{m}}$$
and two positive constants $c_0$, $c_1$ independent of $J$ and $K$ such that
$$|\{\lambda_{\bar{\nu},\bar{m}}^*: \lambda_{\bar{\nu},\bar{m}}^*\not =0 \}|\leq c_02^JJ^{d-1}$$
and 
$$\|\lambda-S_K^*\lambda| s^{0,\Omega}_{p_2,2}f\| \leq c_1 2^{-Jt}J^{(d-1)(\frac{1}{2}-\frac{1}{p_1})}.$$
In addition, $\lambda_{\bar{\nu},\bar{m}}^*$ depends continuously on $\lambda \in U(s^{t,\Omega}_{p_1, p_1}b) $.
\end{proposition}
\begin{proof} We fix $J\in \mathbb{N}$.\\
{\it Step 1.} For arbitrary $K>J$ we define
$$\varLambda_{\mu}=\{(\bar{\nu},\bar{m})\in \nabla_{\mu}: |\lambda_{\bar{\nu},\bar{m}}|\geq \epsilon_{\mu}\},$$
where
$$
\epsilon_{\mu}=
\begin{cases}
0& \text{if }\ 0\leq \mu\leq J,\\
2^{\mu\alpha}2^{J\beta}\mu^{-\frac{d-1}{p_1}}&\text{if }\ J<\mu\leq K
\end{cases}
$$
and
$$ \alpha=-t+\frac{1}{p_1}+\vartheta,\qquad \beta=-\frac{1}{p_1}-\vartheta ,\qquad \vartheta=\frac{1}{2}\frac{t-\frac{1}{p_1}+\frac{1}{\delta_2}}{1-\frac{p_1}{\delta_2}}.$$
Observe that $\vartheta>0$. For $J<\mu\leq K$, this leads to the estimate
\beqq
|\varLambda_{\mu}|&=&\sum_{(\bar{\nu},\bar{m})\in \varLambda_{\mu}}1 \leq \sum_{(\bar{\nu},\bar{m})\in \varLambda_{\mu}}\frac{|\lambda_{\bar{\nu},\bar{m}}|^{p_1}}{\epsilon_{\mu}^{p_1}} \\
&\leq&\epsilon_{\mu}^{-p_1}2^{-\mu(t-\frac{1}{p_1})p_1}\sum_{|\bar{\nu}|=\mu}2^{\mu(t-\frac{1}{p_1})p_1}\sum_{\bar{m}\in A_{\bar{\nu}}^{\Omega}}|\lambda_{\bar{\nu},\bar{m}}|^{p_1}\\
& =& \epsilon_{\mu}^{-p_1}2^{-\mu(t-\frac{1}{p_1})p_1}\| id_{\mu}^*\lambda| (s_{p_1,p_1}^{t,\Omega}b)_{\mu}\|^{p_1}\\
& \leq& \epsilon_{\mu}^{-p_1}2^{-\mu(t-\frac{1}{p_1})p_1} ,
\eeqq
 
\noindent
where $id_{\mu}^*$ denotes the canonical projection of $s_{p_1,p_1}^{t,\Omega}b$ onto $(s_{p_1,p_1}^{t,\Omega}b)_{\mu}$. Inserting the definition of $\epsilon_{\mu}$ into the last inequality we obtain 
$$ |\varLambda_{\mu}|\lesssim 2^J 2^{\vartheta p_1(J-\mu)}\mu^{(d-1)}.$$
Consequently
\beqq
\sum_{\mu=0}^{K}|\varLambda_{\mu}| 
& \lesssim& \sum_{\mu=0}^{J}|\nabla_{\mu}|+ \sum_{\mu=J+1}^{K}|\varLambda_{\mu}|\\
& \lesssim& 2^JJ^{d-1}+ 2^J\sum_{\mu=J+1}^{K}2^{\vartheta p_1(J-\mu)}\mu^{(d-1)}\\
& \lesssim & 2^JJ^{d-1},
\eeqq
with the constant independent of $K>J$.\\
{\it Step 2.} We define the function $\varphi_{\mu}: \C\to \C$ as follows. If $0\leq \mu\leq J$ we put
$$\varphi_{\mu}(z)=z,\ \ \text{ for all }z\in \C, $$
 if $J< \mu\leq K$ and $z=|z|e^{i\theta} \in \C$, then

 \be\label{phi-z}
\varphi_{\mu}(z)= \left\{
 \begin{array}{lll}
 0 & \text{ if }& 0\leq |z|\leq \epsilon_{\mu},\\
 (2|z|-2\epsilon_{\mu})e^{i\theta} & \text{ if }& \epsilon_{\mu}<|z|\leq 2\epsilon_{\mu},\\
 z & \text{ if }& 2\epsilon_{\mu}<|z|.
 \end{array}
 \right.
 \ee
Note that $\varphi_{\mu}$ is a continuous function on $\C$ and satisfies the inequality
\be
|\varphi_{\mu}(z)-z|\leq |z|,\ \ \forall z\in \C \ \  \text{and} \ 0\leq \mu\leq K.\label{bdt5}
\ee
We define
$$ S_K^*\lambda=\sum_{\mu=0}^{K}\sum_{(\bar{\nu},\bar{m})\in \nabla_{\mu}}\lambda_{\bar{\nu},\bar{m}}^*e^{\bar{\nu},\bar{m}},$$
and
$$\lambda_{\bar{\nu},\bar{m}}^*=\varphi_{\mu}(\lambda_{\bar{\nu},\bar{m}}),\ \ (\bar{\nu},\bar{m})\in \nabla_{\mu},\ \mu=0,...,K.$$
The functionals $\lambda_{\bar{\nu},\bar{m}}^*$ depend continuously on $\lambda\in U( s^{t,\Omega}_{p_1, p_1}b)$ since both, $\varphi_{\mu}$ and $\lambda_{\bar{\nu},\bar{m}}$, depend continuously on $\lambda$. \\
{\it Step 3.} Now we estimate the norm $\|\lambda-S_K^*\lambda| s^{0,\Omega}_{p_2,2}f\| $. We define
$$T_{\mu}=\sum_{(\bar{\nu},\bar{m})\in \nabla_{\mu}}(\lambda_{\bar{\nu},\bar{m}}-\lambda_{\bar{\nu},\bar{m}}^*)e^{\bar{\nu},\bar{m}},\ \ 0\leq \mu\leq K.$$
Note that $T_{\mu}=0$ if $0\leq \mu\leq J$. This yields
\beq\label{ws-1}
\|\lambda-S_K^*\lambda|s^{0,\Omega}_{p_2,2}f\|&\leq &
\|\lambda-S_K\lambda|s^{0,\Omega}_{p_2,2}f \|+\|
S_K \lambda-S_K^*\lambda|s^{0,\Omega}_{p_2,2}f \|\\ \nonumber
&\leq& \|\lambda-S_K\lambda|s^{0,\Omega}_{p_2,2}f \|+ \sum_{\mu=J+1}^{K}\|T_{\mu}| (s^{0,\Omega}_{p_2,2}f)_{\mu}\|.
\eeq
For $J+1\leq \mu\leq K$ we have
\beqq
\|T_{\mu} | (s^{0,\Omega}_{p_2, 2}f)_{\mu} \|&\leq& \mu^{(d-1)(\frac{1}{2}-\frac{1}{\delta_2})} \|T_{\mu}| (s_{\delta_2,\delta_2}^{0,\Omega}f)_{\mu} \|\\
&=&\mu^{(d-1)(\frac{1}{2}-\frac{1}{\delta_2})}2^{-\frac{\mu}{\delta_2}}\Big(\sum_{(\bar{\nu},\bar{m})\in \nabla_{\mu}}|\lambda_{\bar{\nu},\bar{m}}-\lambda_{\bar{\nu},\bar{m}}^*|^{\delta_2}\Big)^{\frac{1}{\delta_2}}.
\eeqq
Since $\lambda_{\bar{\nu},\bar{m}}=\lambda_{\bar{\nu},\bar{m}}^*$ whenever $|\lambda_{\bar{\nu},\bar{m}}|>2\epsilon_{\mu}$, we obtain
$$|\lambda_{\bar{\nu},\bar{m}}-\lambda_{\bar{\nu},\bar{m}}^*|^{\delta_2}\leq |\lambda_{\bar{\nu},\bar{m}}|^{\delta_2}=|\lambda_{\bar{\nu},\bar{m}}|^{\delta_2-p_1}|\lambda_{\bar{\nu},\bar{m}}|^{p_1}\leq (2\epsilon_{\mu})^{\delta_2-p_1}|\lambda_{\bar{\nu},\bar{m}}|^{p_1} ,$$
see \eqref{bdt5} and \eqref{phi-z}. Therefore
\beqq
\|T_{\mu} | (s^{0,\Omega}_{p_2, 2}f)_{\mu} \| 
&\lesssim& \mu^{(d-1)(\frac{1}{2}-\frac{1}{\delta_2})}2^{-\frac{\mu}{\delta_2}} (\epsilon_{\mu})^{1-\frac{p_1}{\delta_2}}\Big(\sum_{(\bar{\nu},\bar{m})\in \nabla_{\mu}}|\lambda_{\bar{\nu},\bar{m}}|^{p_1}\Big)^{\frac{1}{\delta_2}}\\
&= &\mu^{(d-1)(\frac{1}{2}-\frac{1}{\delta_2})}2^{-\frac{\mu}{\delta_2}} (\epsilon_{\mu})^{1-\frac{p_1}{\delta_2}}2^{-\mu(t-\frac{1}{p_1})\frac{p_1}{\delta_2}}\|id_{\mu}^*\lambda| (s_{p_1,p_1}^{t,\Omega}b)_{\mu} \|^{\frac{p_1}{\delta_2}}\\
& \lesssim& \mu^{(d-1)(\frac{1}{2}-\frac{1}{\delta_2})} (\epsilon_{\mu})^{1-\frac{p_1}{\delta_2}}2^{-\mu t\frac{p_1}{\delta_2}} .
\eeqq
Employing the definition of $\epsilon_{\mu}$ we obtain
$$\|T_{\mu} | (s^{0,\Omega}_{p_2, 2}f)_{\mu} \| \lesssim
\mu^{(d-1)(\frac{1}{2}-\frac{1}{\delta_2})}  (2^{\mu\alpha}2^{J\beta}\mu^{-\frac{d-1}{p_1}})^{1-\frac{p_1}{\delta_2}}2^{-\mu t\frac{p_1}{\delta_2}} .
$$
Observe that
$$\alpha\Big(1-\frac{p_1}{\delta_2}\Big)=-\frac{1}{2}\Big(t-\frac{1}{p_1}+\frac{1}{\delta_2}\Big)+\frac{tp_1}{\delta_2}\ \text{ and }\ \  \beta\Big(1-\frac{p_1}{\delta_2}\Big)= -\frac{1}{2}\Big(t+\frac{1}{p_1}-\frac{1}{\delta_2}\Big).$$
This results in the estimate
$$\|T_{\mu} | (s^{0,\Omega}_{p_2, 2}f)_{\mu} \| \lesssim
 2^{-\frac{J}{2}(t+\frac{1}{p_1}-\frac{1}{\delta_2})}\mu^{(d-1)(\frac{1}{2}-\frac{1}{p_1})} 2^{-\frac{\mu}{2}(t-\frac{1}{p_1}+\frac{1}{\delta_2})}.$$
Summarizing we have found
\beqq
\sum_{\mu=J+1}^{K}\|T_{\mu} | (s^{0,\Omega}_{p_2, 2}f)_{\mu} \| &\lesssim& 2^{-\frac{J}{2}(t+\frac{1}{p_1}-\frac{1}{\delta_2})}\sum_{\mu=J+1}^{K} \mu^{(d-1)(\frac{1}{2}-\frac{1}{p_1})} 2^{-\frac{\mu}{2}(t-\frac{1}{p_1}+\frac{1}{\delta_2})}\\
&\lesssim& 2^{-Jt}J^{(d-1)(\frac{1}{2}-\frac{1}{p_1})},
\eeqq
with the constant independent of $K>J$. Inserting this into (\ref{ws-1}) and employing Lemma \ref{up-width} we finish our proof. 
\end{proof}
\begin{remark}
\rm The proof given above is a combination of ideas from \cite[Proposition 5.4]{HS3} and \cite{Devo1}.
\end{remark}
\begin{theorem}\label{width2}
Let $0<p_1\leq \infty$, $0<p_2<\infty$ and $t>(\frac{1}{p_1}-\frac{1}{\delta_2})_+$. Then
$$w_n(U(s^{t,\Omega}_{p_1, p_1}b), s^{0,\Omega}_{p_2,2}f)\lesssim n^{-t}\log n^{(d-1)(t-\frac{1}{p_1}+\frac{1}{2})}$$
for all $n\geq 2$.
\end{theorem}
\begin{proof} The assertion for the case $p_1\geq \delta_2$ is an immediate consequence of Corollary \ref{appro} and $w_n\leq a_n$. For the case $p_1< \delta_2$ we follow the argument in \cite{Devo1}. For each $n_J=[c_02^JJ^{d-1}]$, $J\in \mathbb{N}$ and $K=K(J)$ as in Lemma \ref{up-width}, we denote $k=2\sum_{\mu=0}^{K}|\nabla_{\mu}|$. We associate to each $\lambda\in U(s^{t,\Omega}_{p_1,p_1}b)$ a vector $G(\lambda)\in \mathbb{R}^{k}$ whose coordinates are the real and the imaginary part of the coefficients $\lambda_{\bar{\nu},\bar{m}}^*(\lambda)$ in $S^*_K\lambda$. The coefficients appear in $G(\lambda)$ in lexicographic order. Here the real and the imaginary part of the coefficients $\lambda_{\bar{\nu},\bar{m}}^*(\lambda)$ are adjacent. Each vector $G(\lambda)$ has at most $c_0n$ non-zero coordinates, see Proposition \ref{non2}. Since the set $U(s^{t,\Omega}_{p_1, p_1}b) $ is bounded in $s^{0,\Omega}_{p_2,2}f$, there exists a constant $A>0$ such that
$$|\lambda^*_{\bar{\nu},\bar{m}}(\lambda)|\leq |\lambda_{\bar{\nu},\bar{m}}(\lambda)|\leq 2^{\frac{|\bar{\nu}|}{p_2}}\| \lambda|s^{0,\Omega}_{p_2,2}f\| \leq A2^{\frac{K}{p_2}},\ \ \forall \lambda\in U(s^{t,\Omega}_{p_1, p_1}b).$$
Let us denote by $e_i$ the coordinate vectors in $\mathbb{R}^{k}$. Then we define $x_i=Be_i$ and $x_{-i}=-Be_i$, where $B>0$. We consider the set of points 
$$\mathcal{N}=\{x_{-k},...,x_{-1},0,x_1,....,x_k\}.$$ 
Let $\mathcal{T}$ be the collection of all sets $T\subset \{-k,...,-1,0,1,...,k\}$ such that the elements in $T$ are in general position and $|T|\leq c_0n$. Then $\mathscr{C}=\mathscr{C}(\mathcal{N};\mathcal{T})$ is a complex of dimension $c_0n-1$ in $\mathbb{R}^{k}$. If we choose $B$ large enough, then for each $\lambda\in U(s^{t,\Omega}_{p_1, p_1}b)$ there exists an $T\in \mathcal{T}$ such that $G(\lambda)\in \text{conv}\{x_j\}_{j\in T}$. That means, $G$ is a continuous function mapping $U(s^{t,\Omega}_{p_1, p_1}b)$ into $\mathscr{C}$. If $\lambda_1, \lambda_2\in U(s^{t,\Omega}_{p_1, p_1}b)$ and $G(\lambda_1)=G(\lambda_2)$, that is $S^*_K\lambda_1=S^*_K\lambda_2$, from Proposition \ref{non2} we obtain
\beqq
\|\lambda_1-\lambda_2| s^{0,\Omega}_{p_2,2}f \|&\leq & \|\lambda_1-S^*_K\lambda_1| s^{0,\Omega}_{p_2,2}f\|+\|\lambda_2-S^*_K\lambda_2| s^{0,\Omega}_{p_2,2}f\|\\
&\lesssim& 2^{-Jt}J^{(d-1)(\frac{1}{2}-\frac{1}{p_1})}.
\eeqq
Consequently $$\text{diam}(G^{-1}(G(a)))\lesssim 2^{-Jt}J^{(d-1)(\frac{1}{2}-\frac{1}{p_1})}.$$
This implies
$$ a^{n_J}(U(s^{t,\Omega}_{p_1, p_1}b)|s^{0,\Omega}_{p_2,2}f) \lesssim 2^{-Jt}J^{(d-1)(\frac{1}{2}-\frac{1}{p_1})}.$$
From the monotonicity of the nonlinear co-widths we conclude
\[
a^{[c_02^JJ^{d-1}]}(U(s^{t,\Omega}_{p_1, p_1}b)|s^{0,\Omega}_{p_2,2}f) \lesssim \log (c_0  \, J^{d-1}\, 2^J)^{(d-1)(\frac{1}{2}-\frac{1}{p_1})}\, 
\Big(\frac{c_0 \, J^{d-1}\, 2^J}{\log^{d-1} (c_0  \, J^{d-1}\, 2^J)}\Big)^{- t} \, .
\]
Again the monotonicity of the nonlinear co-widths allows us to switch 
from the subsequence $([c_0  \, J^{d-1}\, 2^J])_J$ to $n\in \N$ in this formula by possibly changing the constant behind $\lesssim$. Finally, in view of the equivalence of the nonlinear widths, see Lemma \ref{equiv}, the claim follows.
\end{proof}

\subsection{Bernstein numbers of sequence spaces related to spaces of dominating mixed smoothness}\label{bern}
In this subsection we investigate the asymptotic behaviour of Bernstein numbers of the embedding $$id^*:s^{t,\Omega}_{p_1, p_1}b \rightarrow s^{0,\Omega}_{p_2, 2}f. $$
\begin{lemma}
For all $\mu \in \N_0$ and all $n \in \N$ we have 
\begin{equation}\label{eqlow1}
b_n(id_{\mu}^*:(s^{t,\Omega}_{p_1, p_1}b)_\mu \rightarrow (s^{0,\Omega}_{p_2, 2}f)_\mu)\leq b_n(id^*:s^{t,\Omega}_{p_1, p_1}b \rightarrow s^{0,\Omega}_{p_2, 2}f) \, .
\end{equation}
\end{lemma}

\begin{proof}
We consider the following diagram
\be\label{f5}
\begin{CD}
s^{t,\Omega}_{p_1,p_1}b @ > id^* >> s^{0,\Omega}_{p_2,2}f\\
@A id^1 AA @VV id^2 V\\
(s^{t,\Omega}_{p_1,p_1}b)_\mu @ > id_{\mu}^* >> (s^{0,\Omega}_{p_2,2}f)_\mu \, .
\end{CD}
\ee
Here $id^1$ is the canonical embedding and $id^2$ is the canonical projection.
Since $id_{\mu}^* = id_2 \circ id^* \circ id_1$ the property $(s3)$ yields
$$
b_n(id_{\mu}^*) \leq \| \, id^1\, \| \, \|\, id^2 \, \| \, b_n(id^*) = b_n(id^*)\, .
$$
This completes the proof.
\end{proof}

\begin{lemma}
Let $0<p_1\leq \infty$, $t\in \mathbb{R}$ and $\mu\in \mathbb{N}_0$
\begin{enumerate}
\item If $0< p_2\leq 2$ it holds
\begin{equation}\label{case1}
\mu^{(d-1)(-\frac{1}{p_2}+\frac{1}{2})}2^{\mu(-t+\frac{1}{p_1}-\frac{1}{p_2})}b_n(id_{p_1,p_2}^{D_{\mu}} )\lesssim b_n(id^*).
\end{equation}
\item If $2\leq p_2<\infty$ and $p\in [2,p_2]$ it holds
\be\label{case4}
\ 2^{\mu( -t+\frac{1}{p_1} -\frac{1}{p})}b_n(id_{p_1,p}^{D_{\mu}} )\lesssim b_n(id^*).
\ee
\end{enumerate}
\end{lemma}

\begin{proof}
{\em Step 1.} Proof of (i). We use the following commutative diagram

\tikzset{node distance=4cm, auto}

\begin{center}
\begin{tikzpicture}
 \node (H) {$(s^{t,\Omega}_{p_1,p_1}b)_\mu $};
 \node (L) [right of =H] {$(s^{0,\Omega}_{p_2,p_2}f)_\mu$};
 \node (L2) [right of =H, below of =H, node distance = 2cm ] {$(s^{0,\Omega}_{p_2,2}f)_\mu$};
 \draw[->] (H) to node {$id^2$} (L);
 \draw[->] (H) to node [swap] {$id_\mu^*$} (L2);
 \draw[->] (L2) to node [swap] {$id^1$} (L);
 \end{tikzpicture}
\end{center}

\noindent
This time we have $ b_n(id^2)\leq \|\, id^1\, \|\, b_n(id_{\mu}^*) $. By Lemma \ref{ba2-1}, we obtain
$$ \| \, id^1\, \| \lesssim \mu^{(d-1)(\frac{1}{p_2}-\frac{1}{2})}\, . $$
Lemma \ref{ba1} (ii) yields
\[ 
b_n(id^2)\gtrsim 2^{\mu( -t+\frac{1}{p_1} -\frac{1}{p_2})} \, b_n (id_{p_1,p_2}^{D_{\mu}})\, . 
\]
Inserting this in the previous estimate we find
\begin{equation}\label{low1}
b_n(id^*_{\mu})\gtrsim \mu^{(d-1)(\frac{1}{2}-\frac{1}{p_2})} \, 2^{\mu( -t+\frac{1}{p_1} -\frac{1}{p_2})}
\, b_n (id_{p_1,p_2}^{D_{\mu}})\, .
\end{equation}
From (\ref{eqlow1}) we obtain \eqref{case1}.\\
{\em Step 2.} Let $2\leq p\leq p_2$. We consider the diagram

\tikzset{node distance=4cm, auto}

\begin{center}
\begin{tikzpicture}
 \node (H) {$(s^{t,\Omega}_{p_1,p_1}b)_\mu $};
 \node (L) [right of =H] {$(s^{0,\Omega}_{p,p}f)_\mu $};
 \node (L2) [right of =H, below of =H, node distance = 2cm ] {$(s^{0,\Omega}_{p_2,2}f)_\mu$};
 \draw[->] (H) to node {$id^2$} (L);
 \draw[->] (H) to node [swap] {$id^*_\mu$} (L2);
 \draw[->] (L2) to node [swap]{$id^1$} (L);
 \end{tikzpicture}
\end{center}
\noindent
Because of $ b_n(id^2)\leq \| \, id^1\, \| \, b_n(id_{\mu}^*) $, $ \| \, id^1\, \| \lesssim 1$, see Lemma \ref{ba2-1}, 
and 
$$ b_n(id^2)\gtrsim 2^{\mu( -t+\frac{1}{p_1} -\frac{1}{p})} \, b_n(id_{p_1,p}^{D_\mu}) \, , $$
see Lemma \ref{ba1} (ii), we obtain
$$
b_n(id_{\mu}^*)\gtrsim 2^{\mu( -t+\frac{1}{p_1} -\frac{1}{p})}\, b_n(id_{p_1,p}^{D_\mu})\, . 
$$
Now employing (\ref{eqlow1}), we have inequality (\ref{case4}). The proof is complete.
\end{proof}
To proceed we recall the behaviour of Weyl numbers $x_n(id^*: s_{p_1,p_1}^{t,\Omega} b\rightarrow s_{p_2,2}^{0,\Omega}f)$, see \cite{KiSi}.

\begin{proposition}\label{weyl2}
Let $1\leq p_1\leq \infty$, $1\leq p_2<\infty$ and $t>(\frac{1}{p_1}-\frac{1}{p_2})_+$. Then
$$x_n(id^*: s_{p_1,p_1}^{t,\Omega} b\rightarrow s_{p_2,2}^{0,\Omega}f)\asymp \varphi_2(n,t,p_1,p_2)=n^{-\alpha}(\log n)^{(d-1)\beta},\ \ n\geq 2,$$
where


\begin{enumerate}
\item{\makebox[5.3cm][l]{$\alpha=t$, $\beta=0$} if\ \  $p_1,p_2\leq 2,\ t< \frac{1}{p_1}-\frac{1}{2} $},
\item{\makebox[5.3cm][l]{$\alpha=t$, $\beta=t+\frac{1}{2}-\frac{1}{p_1}$} if\ \  $p_1,p_2\leq 2 \text{,\ }t> \frac{1}{p_1}-\frac{1}{2}$},
\item{\makebox[5.3cm][l]{$\alpha=t-\frac{1}{2}+\frac{1}{p_2}$, $\beta=t+\frac{1}{p_2}-\frac{1}{p_1}$} if\ \  $p_1\leq 2\leq p_2$},
\item{\makebox[5.3cm][l]{$\alpha=\beta=t-\frac{1}{p_1}+\frac{1}{2} $} if\ \  $p_2\leq2\leq p_1,\ t>\frac{1}{p_1} $ },
\item{\makebox[5.3cm][l]{$\alpha=\beta=t-\frac{1}{p_1}+\frac{1}{p_2} $} if\ \  $ \Big(2\leq p_2\leq p_1,\ t>\frac{\frac{1}{p_2}-\frac{1}{p_1}}{\frac{p_1}{2}-1}\Big)$ or $ 2\leq p_1\leq p_2$},
\item{\makebox[5.3cm][l]{$\alpha=\frac{tp_1}{2} $, $\beta= t+\frac{1}{2}-\frac{1}{p_1}$} if\ \  $\Big(2\leq p_2< p_1,\ t<\frac{\frac{1}{p_2}-\frac{1}{p_1}}{\frac{p_1}{2}-1}\Big) \text{ or } \big(p_2\leq2<p_1,\ t<\frac{1}{p_1}\big).$}

\end{enumerate}
\end{proposition}
As a consequence of Proposition \ref{weyl2} and Lemma \ref{bern-weyl0} we obtain the following
\begin{corollary}
\label{bern-weyl2}
Let $1\leq p_1\leq \infty$, $1\leq p_2<\infty$ and $t>(\frac{1}{p_1}-\frac{1}{p_2})_+$. Then
$$b_n(id^*: s_{p_1,p_1}^{t,\Omega} b\rightarrow s_{p_2,2}^{0,\Omega}f)\lesssim x_n(id^*: s_{p_1,p_1}^{t,\Omega} b\rightarrow s_{p_2,2}^{0,\Omega}f),\ \ n\geq 2.$$
\end{corollary}
Now we are in position to formulate the main result of this subsection.
\begin{theorem}\label{main2-1}
Let $1\leq p_1\leq \infty$, $1\leq p_2<\infty$ and $t>(\frac{1}{p_1}-\frac{1}{p_2})_+$. Then
$$b_n(id^*: s_{p_1,p_1}^{t,\Omega} b\rightarrow s_{p_2,2}^{0,\Omega}f)\asymp \phi_2(n,t,p_1,p_2),\ \ n\geq 2,$$
where $\phi_2(n,t,p_1,p_2)$ is given in Theorem \ref{main2}.
\end{theorem}
\begin{proof}
{\it Step 1.} Let $\phi_2(n,t,p_1,p_2)$ be the function defined in Theorem \ref{main2}. We prove the lower bound.\\
{\it Substep 1.1.} Proof of (i). We consider the following commutative diagram
\[
\begin{CD}
s^{t,\Omega}_{p_1,p_1}b @ > id^* >> s^{0,\Omega}_{p_2,2}f \\
@A id_1 AA @VV id_2 V\\
2^{\mu(t-\frac{1}{p_1})}\ell_{p_1}^{A_{\mu}} @ > I_{\mu} >> 2^{\mu(0-\frac{1}{p_2})} \ell_{p_2}^{A_{\mu}} \,.
\end{CD}
\]
Here $A_{\mu}=|A_{\bar{\nu}}^{\Omega}|$ for some $\bar{\nu}$ with $|\bar{\nu}|_1=\mu$, $id_1$ is the canonical embedding, whereas $id_2$ is the canonical projection. From $(s3)$ we derive 
\[
b_n(I_{\mu})=b_n(id_2\circ id^* \circ id_1)\leq \| id_1\| \, \|id_2\| \, b_n(id^*)=b_n(id^*)\, .
\]
Also $(s3)$ guarantees
\[
b_n(I_{\mu})= 2^{\mu(-t+\frac{1}{p_1}-\frac{1}{p_2})} \, b_n(id_{p_1,p_2}^{A_{\mu}})\, .
\]
We choose $n=\big[\frac{A_{\mu}}{2}\big]$. Then the inequality (\ref{th3}) yields
\[
b_n(I_{\mu})\geq 2^{\mu(-t+\frac{1}{p_1}-\frac{1}{p_2})} \, b_n(id_{p_1,p_2}^{A_{\mu}})
\gtrsim 2^{\mu(-t+\frac{1}{p_1}-\frac{1}{p_2})} \, 2^{\mu(\frac{1}{p_2}-\frac{1}{p_1})} = 2^{-\mu t} \asymp n^{-t}\, ,
\]
which implies $b_n(id^*)\gtrsim n^{-t}$.\\
{\it Substep 1.2.} We prove $(v)$. In this case we choose $n=[D_{\mu}^{\frac{2}{p_1}}]$. 
If $2\leq p_2\leq p_1$ it follows from property (\ref{th4}) that
\[
b_n(id_{p_1,2}^{D_{\mu}} )\gtrsim D_{\mu}^{\frac{1}{2}-\frac{1}{p_1}} \gtrsim (\mu^{d-1}2^{\mu})^{\frac{1}{2}-\frac{1}{p_1}}\, .
\]
Inequality \eqref{case4} with $p=2$ yields
\[
b_n(id^*) \gtrsim  2^{\mu(-t+\frac{1}{p_1}-\frac{1}{2})}(\mu^{d-1}2^{\mu})^{\frac{1}{2}-\frac{1}{p_1}}\gtrsim \mu^{(d-1)(\frac{1}{2}-\frac{1}{p_1})} \, 2^{-t\mu} \, .
\]
If $p_2\leq 2\leq p_1$ we employ inequalities \eqref{case1}, \eqref{th4} and obtain
\beqq b_n(id^*)&\gtrsim&\mu^{(d-1)(-\frac{1}{p_2}+\frac{1}{2})}2^{\mu(-t+\frac{1}{p_1}-\frac{1}{p_2})}b_n(id_{p_1,p_2}^{D_{\mu}} ) \\
&\gtrsim& \mu^{(d-1)(-\frac{1}{p_2}+\frac{1}{2})}2^{\mu(-t+\frac{1}{p_1}-\frac{1}{p_2})}(\mu^{d-1}2^{\mu})^{\frac{1}{p_2}-\frac{1}{p_1}}\\
&=&\mu^{(d-1)(\frac{1}{2}-\frac{1}{p_1})} \, 2^{-t\mu} 
\eeqq
Rewriting the right-hand side in dependence on $n$ we find
\[
b_n(id^*) \gtrsim n^{-\frac{tp_1}{2}}(\log n)^{(d-1)(t+\frac{1}{2}-\frac{1}{p_1})}\, .
\]
{\it Substep 1.3.} For the lower bound in the cases $(ii)$, $(iii)$ and $(iv)$ we use inequalities (\ref{case1}), (\ref{case4}) combined with (\ref{th1}), (\ref{th2}), \eqref{th3} and choose $n=\big[\frac{D_{\mu}}{2}\big]$. \\
\noindent
{\it Step 2.} To prove the upper bound we use, as in the isotropic case, the inequality 
$$b_n\lesssim \min(x_n,w_n),\ \ n\in \mathbb{N},$$
see Proposition \ref{bern-width} and Corollary \ref{bern-weyl2}. In view of Theorem \ref{width2} and Proposition \ref{weyl2} the claim follows.
\end{proof}
\subsection{Proof of Theorem \ref{main2}}
Now we are ready to prove the Theorem \ref{main2}. We shift the results obtained in Subsection \ref{bern} to the situation of function spaces.\\
{\bf Proof of Theorem \ref{main2}.} First we prove that under the restrictions in Theorem \ref{main2-1} 
\beq\label{ws-2}
b_n( id^*: s_{p_1,p_1}^{t,\Omega}b \to s_{p_2,2}^{0,\Omega}f)\asymp b_n\big(id: S_{p_1,p_1}^t B(\Omega)\to S_{p_2,2}^0 F(\Omega)\big)
\eeq
holds for all $n \in \N$. By $\ce_d $ we denote a linear continuous extension operator from $ S^{t}_{p_1,p_1}B (\Omega)$ to $S^{t}_{p_1,p_1}B (\R)$. For the existence of those operators we refer to \cite{KiSi}, see also Triebel \cite[1.2.8]{Tr10}. We consider the commutative diagram
\[
\begin{CD}
S_{p_1,p_1}^{t}B(\Omega) @ >\mathcal{E}_d >> S_{p_1,p_1}^{t}B(\mathbb{R}^d) @>\mathcal{W}>>s_{p_1,p_1}^{t,\Omega}b\\
@V id VV @. @VV id^*V\\
S_{p_2,2}^{0}F(\Omega) @ <R_{\Omega}<< S_{p_2,2}^{0}F(\mathbb{R}^d)@ <\mathcal{W}^* << s_{p_2,2}^{0,\Omega}f
\end{CD}
\]
Here the mapping $\mathcal{W}$ is defined as 
\[ 
\mathcal{W} f := \, \Big( 2^{|\bar{\nu}|_1}\, \langle f, \, \Psi_{\bar{\nu}, \bar{m}}\rangle
\Big)_{\bar{\nu}\in \N_0^d, \, \bar{m} \in A^{\Omega}_{\bar{\nu}}}\, .
\]
Furthermore, $\mathcal{W}^*$ is defined as
\[
\mathcal{W}^* \lambda := 
\sum_{\bar{\nu} \in \N_0^ d}
\sum_{\bar{m} \in A^{\Omega}_{\bar{\nu}}} \lambda_{\bar{\nu}, \bar{m}} \, \Psi_{\bar{\nu}, \bar{m}}
\]
and $R_\Omega$ means the restriction to $\Omega$. The boundedness of $\ce_d, \mathcal{W}, \mathcal{W}^*, R_\Omega$ and property $(s3)$ yield
$$b_n (id)\lesssim b_n (id^*).$$ 
Modifying the diagram we obtain the inverse inequality as well, see \cite{Vy1}. Now the assertion in Theorem \ref{main2} follows from \eqref{ws-2}, Theorem \ref{main2-1} and the Littlewood-Paley assertion $S_{p_2,2}^0F(\Omega)=L_{p_2}(\Omega)$ (in the sence of equivalent norms, $1<p_2<\infty$), see Nikol'skij \cite[1.5.6]{Ni}. The proof is complete. \qed


\end{document}